\documentclass[11pts]{elsarticle}

\usepackage{hyperref}
\usepackage[hmargin=3cm]{geometry}

\usepackage{ae,lmodern} % ou seulement l'un, ou l'autre, ou times etc.
\usepackage[english]{babel}
\usepackage[latin1]{inputenc}
\usepackage[T1]{fontenc}

\usepackage{environ,tikz}
\usepackage{pgfplots}
\usepgfplotslibrary{polar}
     \usepgflibrary{shapes.geometric}
     \usetikzlibrary{calc}
     \usetikzlibrary{automata,positioning,arrows}

\pgfplotsset{compat=1.3}
\usetikzlibrary{arrows, decorations.markings}
\usetikzlibrary{decorations.text}

\usetikzlibrary{plotmarks}

\usepackage{amsmath,amsfonts,amssymb}%
\usepackage{bbm,enumerate,stackrel,graphicx}
\usepackage{hyperref}
\usepackage{hypcap}
\usepackage{paralist}

\pgfplotsset{my style/.append style={axis x line=middle, axis y line= middle, xlabel={$\gamma$}, ylabel={$v$}, axis equal }}

\usepackage{subcaption}

\usepackage{multimedia}

\def\Lip{{\hbox{\rm Lip}}}

\newcommand{\RR}{\mathbb R}

\newcommand{\NN}{\mathbb N}
\newcommand{\R}{\mathbb R}

\newcommand{\rr}{{\RR_+}}

\newcommand{\ue}{\frac{1}{\varepsilon}}

\newcommand{\dt}{\partial_t}

\newcommand{\e}{\varepsilon}
\newcommand{\dz}{\partial_z}
\newcommand{\dzz}{\partial^2_{z^2}}

\newcommand{\normel}{\left\|}
\newcommand{\normer}{\right\|}

\newcommand{\sgn}{\,{\rm sgn}}

\newcommand{\ddt}[1]{\frac{{\rm d} #1}{\rm d \it t}}
\newcommand{\ti}[1]{\tilde{#1}}

\newcommand{\vertiii}[1]{{\left\vert\kern-0.25ex\left\vert\kern-0.25ex\left\vert #1 
		\right\vert\kern-0.25ex\right\vert\kern-0.25ex\right\vert}}
\newcommand{\nrm}[2]{{\normel{#1}\normer}_{#2}}

\newcommand{\SysNoNumB}[1]{
	\begin{displaymath}
	\left\{
	\begin{array}{#1}
}
\newcommand{\SysNoNumE}{
	\end{array}
	\right.
	\end{displaymath}
}
\newcommand{\TabNoNumB}[1]{
	\begin{displaymath}
	\begin{array}{#1}
}
\newcommand{\TabNoNumE}{
	\end{array}
	\end{displaymath}
}
\newcommand{\TabNumB}[2]{
	\begin{equation}
	\label{#2}
	\begin{array}{#1}
}
\newcommand{\TabNumE}{
	\end{array}
	\end{equation}
}
\newcommand{\SysNumB}[2]{
	\begin{equation}
	\left\{
	\label{#2}
	\begin{array}{#1}
}
\newcommand{\SysNumE}{
	\end{array}
	\right.
	\end{equation}
}

\newcommand{\etau}[1]{}%{\varepsilon^{p_{#1}}}

\newcommand{\chiu}[1]{\ensuremath{\chi_{#1}}}

\newcommand{\cR}{{\mathcal R}}

\newcommand{\cH}{{\mathcal H}}

\newcommand{\ovu}{\ov{U}}

\newcommand{\rhoe}{\rho_\e}

\newcommand{\rhoz}{\rho_0}

\newcommand{\da}{\partial_a} 
\newcommand{\zeps}{{  z}_\e}

\newcommand{\veps}{{ u}_\e}
\newcommand{\vepsi}{{ u}_{I}}

\newcommand{\zp}{{\mathbf  z}_p}

\newcommand{\tit}{\tilde{t}}

\newcommand{\cT}{{\mathcal T}}
\newcommand{\hv}{{\hat{v}_\e}}

\newcommand{\vz}{u_0}

\newcommand{\dtit}{\partial_{\tit}}

\newcommand{\rhoinf}{\rho^{\infty}}

\newcommand{\st}{\text{ s.t. }}

\newcommand{\tia}{{\ti{a}}}

\newcommand{\cE}{{\mathcal E}}

\newcommand{\hz}{{\hat{z}_\e}}

\newcommand{\tse}{{\frac{t}{\e}}}

\newcommand{\zz}{{\mathbf z}_0}
\newcommand{\cL}{{\mathcal L}}

\newcommand{\vepsd}[2]{{\mathbf  U}^{#1}_{\e,#2}}

\newcommand{\cle}{\cL_\e}

\newcommand{\zepsd}[1]{\bfZ^{#1}_{\e}}

\newcommand{\zpeps}{z_{p,\e}}

\DeclareMathOperator*{\argmin}{argmin}
 \renewcommand{\zp}{z_p}
\renewcommand{\zz}{z_0}
\renewcommand{\hv}{\hat{u}}
\newcommand{\dzeps}{\dot{z}_\e}
\renewcommand{\dzz}{\dot{z}_0}
\newcommand{\kernel}{\varrho}

\newcommand{\vinf}{v_\infty}
\newcommand{\zinf}{z_\infty}
\renewcommand{\rhoinf}{\kernel_\infty}

\newtheorem{theorem}{Theorem}[section]
\newtheorem{corom}{Corollary}[section]
\newtheorem{assum}{Assumptions}[section]
\newtheorem{lemm}{Lemma}[section]
\newtheorem{propm}{Proposition}[section]
\newtheorem{rmkm}{Remark}[section]
\newproof{proof}{Proof}
\begin{document}
\begin{frontmatter} % The preamble begins here.

%\begin{CJK*}{G\kernel}{} % Use default fonts from CJK (see below)
\title{Asymptotic limits  for a non-linear integro-differential equation modelling leukocytes' rolling on arterial walls}
%\runtitle{Asymptotic limits for an integro-differential non-linear equation}
%\author{Vuk Milisic} 
%\affiliation{CNRS/Universit\'e Paris 13}
%\author{Dietmar Oelz}
%\affiliation{University of Queensland}
%\author{Christian Schmeiser} 
%\affiliation{Wolfgang Pauli Institute}
\author[1]{Vuk Mili\v{s}i\'c\fnref{fn1}}
\address[1]{CNRS/Laboratoire de Mathématiques de Bretagne Atlantique, UMR 6205, {France}}
\fntext[fn1]{{\tt vuk.milisic@univ-brest.fr}}

\author[2]{Christian Schmeiser\fnref{fn2}}
\address[2]{Fakult\"at f\"ur Mathematik, Universit\"at Wien, %Oskar-Morgenstern- Platz 1, 
	1090 Wien, Austria}
\fntext[fn2]{{\tt christian.schmeiser@univie.ac.at}}

%\author[1]{\inits{V.}\fnms{Vuk} \snm{Milisic}\ead[label=e1]{milisic@math.univ-paris13.fr}%
%	\thanks{Corresponding author,
%	\printead{e1}.}%
%	%\thanks{Do not use capitals for the author's surname.}
%}%,
%and
% \author[2]{\inits{C.}\fnms{Christian} \snm{Schmeiser}\ead[label=e2]{christian.schmeiser@univie.ac.at}
%\thanks{\printead{e2}}}
%\address[1]{CNRS/Laboratoire Analyse, G{\'e}om{\'e}trie \& Applications, Universit{\'e} Paris 13, \cny{France}%\printead[presep={\\}]{e1}
%}
% \address[2]{%Wolfgang Pauli Institute, 
% 	Fakult\"at f\"ur Mathematik, Universit\"at Wien, %Oskar-Morgenstern- Platz 1, 1090 Wien,
% 	  \cny{Austria}}

%\maketitle

%\author[A]{\inits{Laboratoire Analyse, G{\'e}om{\'e}trie \& Applications (LAGA), Universit{\'e} Paris 13,FRANCE} 
%		   \fnms{First} 
%		   \snm{Vuk Mili\v si\' c}\ead[label=e1]{milisic@math.univ-paris13.fr}%
%	\thanks{Corresponding author. \printead{e1}.}}%
%\author[A]{ Vuk Mili\v si\' c \thanks{Laboratoire Analyse, G{\'e}om{\'e}trie \& Applications (LAGA),
%Universit{\'e} Paris 13,
%FRANCE ({\tt milisic@math.univ-paris13.fr}), Draft version of \today.}
%\and 
%Christian Schmeiser 
%\thanks{Fakult\"at fÃ¼r Mathematik, Universit\"at Wien, Oskar-Morgenstern- Platz 1, 1090 Wien, AUSTRIA. 
%({\tt E-mail: Christian.Schmeiser@univie.ac.at}).
%}
%}

%...
%\begin{keyword}
%	\kwd{cell motility}
%	\kwd{Lipschitz mechanical energy}
%	\kwd{delayed gradient flow}
%	\kwd{Volterra integral equations}
%	\kwd{asymptotic limits}
%\end{keyword}

%\tableofcontents

\begin{abstract}{
      We consider a non-linear integro-differential  model describing $z$, the position of
	the cell center on the real line presented in \cite{pmid29571710}. 
	We introduce a
	new $\e$-scaling and we prove rigorously the asymptotics when $\e$ goes to zero. 
	We show that this scaling characterizes the long-time behavior of the solutions
	of our problem in the cinematic regime ({\em i.e.} the velocity $\dot{z}$ tends to a limit).
	The convergence results are first given when $\psi$, the elastic energy associated to linkages, 
	is convex and regular (the second order derivative of $\psi$
	is bounded). 
	In the absence of blood flow,  when $\psi$,
	is quadratic, we compute the final position $\zinf$ to which we prove that $z$ tends.
	We then build a rigorous mathematical framework for $\psi$ being  convex but only Lipschitz.
	We extend convergence results with respect to $\e$ to the case when 
	$\psi'$ admits a finite number of jumps.
	In the last part, we show that in the constant force case (see Model 3 in \cite{pmid29571710}, {\em i.e.}  $\psi$ is the absolute value), we solve explicitly the problem and recover
	 the above asymptotic results.}
\end{abstract}
\end{frontmatter}

%\linenumbers

\setcounter{tocdepth}{1}
\tableofcontents

\section{Introduction}
Neutrophils are the first line of defense against bacteria and fungi and help fighting parasites and viruses. They are
necessary for mammalian life, and their failure to recover after myeloablation is fatal. 
Neutrophils are short-lived, effective killing machines. 
They take their cues directly from the infectious organism, 
from tissue macrophages and other elements of the immune system. 
Neutrophils get close to their destination through the blood system. 
{When receiving chemical signals, they express adhesion molecules \cite{pmid23112187}, responsible for their rolling, slowing down, and eventual sticking to 
vessel walls \cite{pmid1572393} (see Fig. \ref{fig.neutro}), followed by extravasation and crawling through tissue towards their final destination.}
\begin{figure}[ht!]
\begin{center}
 \includegraphics[width=0.9\textwidth]{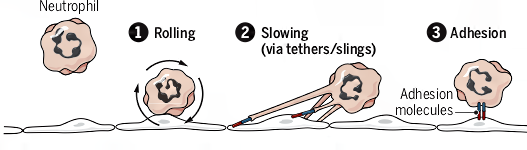}
\end{center}
\caption{A schematic view of the interactions between a  neutrophil and the arterial wall in the blood flow. (illustration taken from \cite{pmid30530726})}\label{fig.neutro}
\end{figure}

In this article we analyze a class of models for the process of rolling and slowing down along the vessel wall by transient elastic linkages. The model has
the nondimensionalized form
\begin{equation}\label{eq.z.nl}
 \begin{aligned}
 & \dzeps(t) + \int_\rr \psi'\left( \frac{\zeps(t)-\zeps(t-\e a)}{\e} \right) \kernel(a,t) da= v(t)  \,, & t \in (0,T] \,,\\
 & \zeps(t) = \zp(t) \,, &t\le 0 \,.
\end{aligned}
\end{equation}
Here $\zeps(t)\in\R$ is the position of the cell at time $t$ with the given past positions $\zp(t)$, $t\le 0$. The integro-differential equation describes a 
force balance between the friction force $f(t) = v(t) - \dzeps(t)$ with the blood flow velocity $v(t)$, and the elastic linkage forces between the cell and the
vessel wall, described by the integral. These forces are parametrized by the age $a$ of the linkages, and the density of linkages with respect to their age
is given by $\rho(a,t)\ge 0$. The function $\psi$ describes the potential energy of a linkage, dependent on the distance $\zeps(t)-\zeps(t-\e a)$ between the present
position of the cell and its position when the linkage has been established (see Fig. \ref{fig.roll.cell.scheme}). The dimensionless parameter $\e>0$ results from scaling 
and represents the ratio between the typical age of a linkage and a characteristic time for the cell movement. Small values of $\e$ correspond to a rapid turnover of 
linkages. The occurrence of the factor $1/\e$ is a scaling assumption, needed to obtain an effect of the linkages in the limit $\e\to 0$. However, \eqref{eq.z.nl} can also
be seen as a macroscopic rescaling of the model for the microscopic unknown $Z(\tau)=\zeps(\e\tau)/\e$. Note that in this interpretation we assume the data $v$
and $\rho$ to vary in terms of the macroscopic time $t$.
 \begin{figure}[ht!]
	\centering\footnotesize
\includegraphics[width=0.9\textwidth]{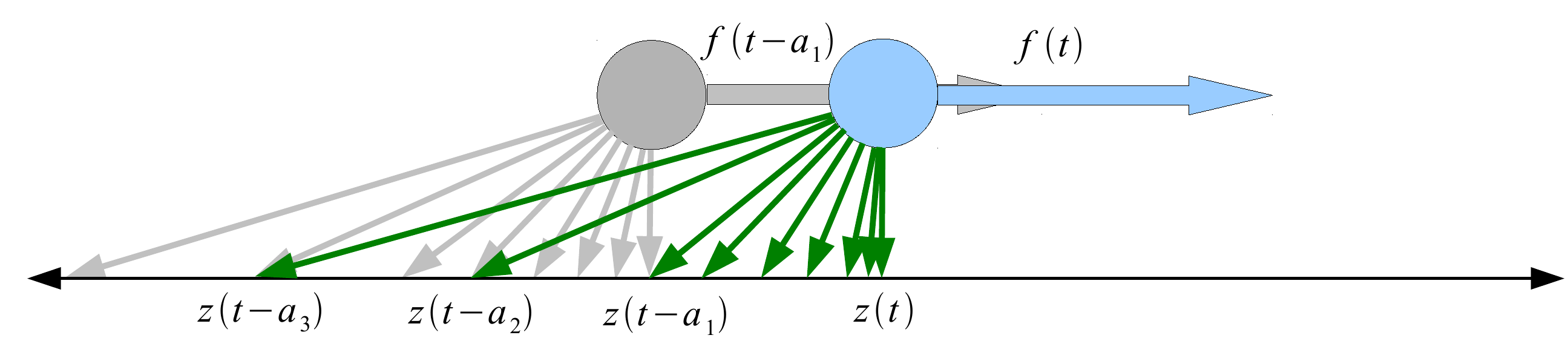}
\caption{The position of the moving binding site at time $t$ and time $t-{a_1}$ with some of the respective linkages.}\label{fig.roll.cell.scheme}
\end{figure}

{Models of the form \eqref{eq.z.nl} with various choices of $\psi$ have been derived in \cite{pmid29571710}, passing from a probabilistic description to an averaged 
version. The simplest example is a linear model with quadratic potential energy $\psi(u) = u^2/2$. This has already been formulated in 1960s together with the 
formal macroscopic limit as a derivation of rubber friction \cite{Schallamach}. It has also been used in the context of the Filament Based Lamellipodium Model \cite{OeSch,MR3385931} for the crosslinking between cytoskeleton filaments and cell-substrate adhesion. There it is usually coupled with an age structured population
model for the density $\rho$. Its mathematical analysis has been developed in \cite{MiOel.1,MiOel.2,MiOel.3,MiOel.4,Mi.5,mi.proc}. \\
Nonlinear models may contain the effects of material or
of geometric nonlinearities. An example for the latter is a model with linkages in the form of membrane tethers \cite{pmid30530726} connecting the anchoring point with the
closest boundary point of a circular cell with (microscopic) radius $r$, see Fig. \ref{fig.pyth}. This gives a tether length $\ell(u) = \sqrt{u^2+r^2}-r$ and, with linear material 
properties, $\psi(u) = \ell(u)^2/2 = O(u^4)$ as $u\to 0$. Concerning material nonlinearities we also allow models with nondifferentiable potentials such as constant force 
$\psi(u)=|u|$. Our main structural assumption is convexity of $\psi$.}

\tikzstyle{vecArrow} = [thick, decoration={markings,mark=at position
   1 with {\arrow[semithick]{open triangle 60}}},
   %double distance=1.4pt, shorten >= 5.5pt,
   preaction = {decorate},
   postaction = {draw,line width=0.4pt}]%, white,shorten >= 4.5pt}]    
   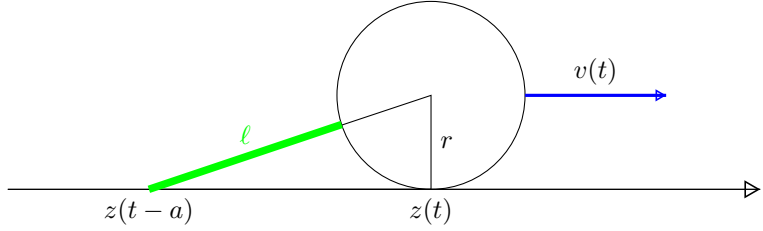
\begin{figure}[ht!]
 	\begin{center}
 		\begin{tikzpicture}[scale=1.25]%,cap=round,>=latex]
		\coordinate [label=below:$z(t-a)$] (A) at (-1.5cm,-1.cm);
		\coordinate [label=below:$z(t)$] (C) at (1.5cm,-1.0cm);
		\coordinate  (D) at (0.55,-0.31);
		\coordinate (B) at (1.5cm,0cm);
		\coordinate (Bp) at (2.5cm,0cm);
		\coordinate (Bpp) at (4cm,0cm);
		\coordinate (E) at (-3cm,-1cm);
		\coordinate (F) at (5cm,-1cm);
		\draw [vecArrow,line width=0.4pt] (E) -| (F);
		\draw (A) -- node[above] {} (B) -- node[right] {$r$} (C) -- node[below] {} (A);
		\draw (1.5,0) circle (1cm);
		\draw[green,
		decoration={markings,mark=at position 1 with },
		postaction={decorate},line width=1.0mm ] %,postaction={decorate,decoration={text along path,text align=center,text={$\psi(z(t)-z(t-a))$}}}] 
(D) -- node[above] {$\ell$}   (A);
		\draw[blue,decoration={markings,mark=at position 1 with {\arrow[scale=0.5,blue]{open triangle 60}}},
		postaction={decorate},very thick, text=black] (Bp)-- node[above] {$v(t)$}(Bpp);
		\end{tikzpicture}
	\end{center}
   \caption{The actual length of filaments of a cell of radius r is the dashed (green) segment whose length is $\ell=\sqrt{(z(t)-z(t-a))^2+r^2}-r
   	$}\label{fig.pyth}
   \end{figure}

{The formal macroscopic limit of \eqref{eq.z.nl} as $\e\to 0$ reads
\begin{equation}\label{eq.zz.nl}
 \begin{aligned}
 & \dzz(t)  + \int_\rr \psi'(a \dzz(t))\kernel(a,t)  da = v(t) \,,  & t>0 \,,\\
 &  \zz(t) =\zp(t) \,, & t\le0 \,.
\end{aligned}
\end{equation}
Convexity of $\psi$ implies that the left hand side of the implicit ODE is a strictly increasing function of $\dzz(t)$.
A rigorous justification of the macroscopic limit has been given in \cite{MiOel.1} for the linear problem without the
additional friction force due to the blood flow. Generalizations of this result belong to the main goals of the present work.\\
Strongly related is the long time asymptotics. Assuming the data $(\rho(a,t),v(t))$ to converge to $(\rho_\infty(a),v_\infty)$ as 
$t\to\infty$, we expect convergence of the velocity $\dzeps(t)$ to a constant $\dot{z}_\infty$ satisfying
\begin{equation}\label{eq.lim.t.large}
	\dot{z}_\infty + \int_\rr \psi'(a \dot{z}_\infty) \kernel_\infty(a) da = v_\infty \,,
\end{equation}
essentially the same equation as in \eqref{eq.zz.nl}. Again we shall be interested in making this limit rigorous.
}

Another concern of this article, motivated by the formal computations
made in \cite{pmid29571710} is to give a rigorous mathematical meaning 
{to \eqref{eq.z.nl} in the case, when $\psi$ is only Lipschitz (as a consequence of convexity), and to justify the asymptotic limits
also in this situation.}

\noindent The main results of this paper can be summarized as follows :
\begin{enumerate}[i)]
	\item 
{For $\psi$ convex and additionally with Lipschitz continuous derivative, a comparison principle for a class of integro-differential equations including \eqref{eq.z.nl}
(proved in Section 2) is used in Section 3 to obtain an a priori estimate allowing to show global existence of a unique solution of \eqref{eq.z.nl}. The comparison principle
is also used for an error estimate in the rigorous justification of the limit $\e\to 0$. Under weak convergence assumptions on the data as $t\to\infty$ we prove
$z(t) = \dot{z}_\infty t + O(1)$ for the solution $z$ of \eqref{eq.z.nl} with $\e=1$, where $\dot{z}_\infty$ is the unique solution of \eqref{eq.lim.t.large}. The asymptotic 
behaviour of the $O(1)$-term remains open in general, except for a simple linear model problem with $\dot{z}_\infty=0$, where the limit of $z(t)$ can be computed 
explicitly.}

	\item 
{In Section 4 the case of convex (and therefore locally Lipschitz) $\psi$ without any further smoothness assumptions is treated, except global Lipschitz continuity.
In this situation a new notion of solution is needed. We take inspiration from gradient flows for nonsmooth energy functionals \cite{Evans.Book} and rewrite the
problem with a smoothed potential as a variational inequality, where we can pass to the nonsmooth limit. The limiting variational inequality 
	\begin{equation}\label{eq.diff.inclusion.z.intro}
	\begin{aligned}
		(v(t)-\dzeps(t))(w-&\zeps(t)) + \e \int_\rr \psi\left(\frac{\zeps(t)-\zeps(t-\e a)}{\e}\right)  \kernel(a,t) da   \\
		& \leq   \e \int_\rr \psi\left(\frac{w-\zeps(t-\e a)}{\e}\right) \kernel(a,t) da \,, \quad \forall w \in \RR \,,
		\end{aligned}
	\end{equation}
is then equivalent
to the differential inclusion
$$
  v(t) - \dzeps(t) \in \partial \int_\rr \e\, \psi\left( \frac{\zeps(t)-\zeps(t-\e a)}{\e} \right) \kernel(a,t) da  \,, 
$$
where the right hand side is the subdifferential of the integral interpreted as a function of $\zeps(t)$. We prove global existence of a solution in this sense.
With $w=\zeps(t) + \e \hat w$, the variational inequality is written in a form where we can pass to the limit $\e\to 0$, giving
	\begin{equation}\label{eq.zz.nl.lip}
		\begin{aligned}
		(v-\dzz(t)) & \hat w + \int_{\rr}  \psi( a \dzz(t)) \kernel(a,t) da 
		  \leq \int_{\rr} \psi( a \dzz(t)+\hat w ) \kernel(a,t) da , \quad \forall \hat w \in \RR \,.
		\end{aligned}
	\end{equation} 
The linearization approach of Section 3 for the rigorous limit does not work in the nonsmooth case. However, convergence can be proved under the additional
assumptions of time-independent data $(\rho,v)$, finitely many discontinuities of $\psi'$, and a nonvanishing limiting velocity. The proof relies on the fact that,
by the nonvanishing velocity, the argument of $\psi'$ is close to the discontinuities only for a small set of values of $a$. 
We then extend this result to data $(\kernel_\e,v_\e)$ non-constant in time but whose $\e$-limit pair $(\kernel_0,v_0)$ is constant.
Finally the convergence as $t\to\infty$
is transformed to the convergence as $\e\to 0$ by a rescaling, allowing to apply the previous result. This gives essentially that $z(t) = \dot{z}_\infty t + o(t)$, i.e. a weaker
result than for smooth potentials, where $\dot{z}_\infty$ is equal to the solution $\gamma$ of
	\begin{equation}\label{eq.zz.nl.lip.cst}
		(v- \gamma)w + \int_\rr \psi(a \gamma ) \kernel(a) da \leq \int_\rr \psi(w + a\gamma )  \kernel(a)da \,,\quad \forall w \in \RR \,.
	\end{equation}
}

	\item In order to illustrate our results, we  consider {in Section 5} the case when $\psi (u)= |u|$, 
	and study solutions of \eqref{eq.zz.nl.lip}. We show a {\em plastic} asymptotic behavior of the model~:
	if $\vinf \notin (-\mu_\infty,\mu_\infty)$ where $\mu_\infty := \int_\rr \rhoinf (a) da$, then $\gamma +\mu_\infty \sgn(\gamma) = \vinf$ and $z \sim \gamma t$ when $t$ is large. 
	If $\vinf \in  [-\mu_\infty,\mu_\infty]$, the unique solution of \eqref{eq.zz.nl.lip} is $\gamma=0$  : the neutrophil should stop. 
	In this latter case, the previous asymptotic results do not prove that actually $\dot{z}$ vanishes for
	$t$ growing large. 
	Assuming that $\kernel(a,t):=\rhoinf(a) \chiu{\{a<t\}}(a,t)$ with $\rhoinf$ being a decreasing integrable function and $\chiu{\{a<t\}}(a,t)$ the characteristic function of the set $\{a<t\}$, we show that 
	$$
		z(t) = \begin{cases}
	z^0+ \int_0^t[\vinf-\mu_\infty(\tau)]_+ d\tau, & \text{if} \; \vinf \geq 0,\\
	z^0+ \int_0^t[\vinf+\mu_\infty(\tau)]_- d\tau, & \text{if} \; \vinf \leq 0. 
	\end{cases}
	$$
	where $\mu_\infty(t):=\int_{0}^{t} \rhoinf(t) dt$ and $[\cdot]_\pm$ denotes the positive/negative part. 
	The same approach gives an explicit profile of $z(t)$ in the case when $\vinf \notin [-\mu_\infty,\mu_\infty]$.
	All these arguments provide  rigorous mathematical justifications 
	of numerical observations and formal computation in \cite[Section 3.3.2]{pmid29571710}.
\end{enumerate}

\section{{Notations, generic hypotheses, and a comparison principle}}

{We introduce some notation for the rest of this article. For the final time $T\in (0,\infty]$ we introduce 
$I_T:=[0,T]$ for $T<\infty$ and $I_T:=[0,\infty)$ for $T=\infty$}. {For functional space we write}
$L^p_t L^q_a := L^p(I_T;L^q(\rr))$ for any real $(p,q) \in [1,\infty]^2$, 
and similarly  $L^\infty_{a,t} := L^\infty(\rr\times I_T)$.
The weighted $L^p$ space of functions of $a\in\rr$ with non-negative weight $\omega(a)$ is denoted by $L^p(\omega(a)da)$, $1\le p\le \infty$. 

We state the basic hypotheses that are common to 
results presented hereafter. Extra hypotheses will be
assumed locally in the claims.

\begin{assum}\label{hypo.data}
{For $0<T\le\infty$ we assume that }
	\begin{compactenum}[i)]
	        \item {The potential $\psi$ is even, convex, and $\psi(0)=0\le\psi(u)$, $u\in\R$.}
		\item The past {data $\zp$ is bounded and} Lipschitz on $\RR_-$, {\em i.e.},
		$$
		| \zp(a_1) - \zp(a_2) | \leq L_p |a_1 -a_2|, \qquad (a_1,a_2) \in \RR_-^2 \,.
		$$
		\item The source term {satisfies $v\in C^1(I_T)$}.
		\item The {nonnegative} kernel {satisfies $\kernel\in C_B(I_T;L^1((1+a^2)da))$}.
	\end{compactenum}
\end{assum}

{For later use we prove a comparison principle and a stability estimate for a class of integro-differential equations including \eqref{eq.z.nl}.
\begin{lemm}\label{lemm:comp}
Let $\e,T>0$ and let $\phi(a,t,u)$ be measurable with respect to $(a,t)$,
and let it be odd and nondecreasing as a function of $u$. Let the operator $\cH$ be defined by
$$
   \cH[z](t) := \dot z(t) + \int_0^\infty \phi\left(a,t,\frac{z(t)-z(t-\e a)}{\e}\right)da \,,\qquad 0<t\le T\,,
$$
acting on functions $z$, whose values on $(-\infty,0]$ are prescribed. 
Then $\cH$ satisfies the comparison principle
$$
    \Bigl(\cH[z](t)\ge 0 \mbox{ for } t>0 \Bigr)\quad\mbox{and}\quad \Bigl(z(t)\ge 0 \mbox{ for } t\le 0\Bigr) \qquad\Longrightarrow\qquad z\ge 0 \,.
$$
Any solution $z\in C([0,T])$ of the problem
$$
   \cH[z](t) = f(t) \,,\quad t>0 \,;\qquad z(t) = z_p(t) \,,\quad t\le 0 \,,
$$
satisfies 
$$
    |z(t)| \le \sup_{(-\infty,0)} |z_p| + \int_0^t |f(\tau)|d\tau \,,\qquad 0\le t\le T \,.
$$
\end{lemm}
\begin{proof}
The comparison principle is, as usual, first shown for the case of strict inequalities. Thus, we assume $\cH[z](t)>0$, $t>0$, and $z(t)>0$, $t\le 0$. Let $t_0>0$ denote
the smallest zero of $z$. Then we arrive at the contradiction
$$
    \dot z(t_0) > -\int_0^\infty \phi\left(a,t_0,\frac{-z(t_0-\e a)}{\e}\right) da \ge 0 \,,
$$
implying $z>0$. The statement with non-strict inequalities is obtained in the standard way by an approximation argument:
For $\delta>0$ let $z_\delta := z + \delta(1+t_+)$. This implies
$$
   \cH[z_\delta](t) \ge \delta + \cH[z](t) \ge \delta > 0 \,,\quad t>0 \,,\qquad z_\delta(t) \ge \delta > 0 \,,\quad t\le 0 \,,
$$
giving $z \ge -\delta (1+t_+)$ by the argument with strict inequalities and $z\ge 0$ in the limit $\delta\to 0$.\\
Finally we define $Z(t) := \sup_{\tau\le 0}|z_p(\tau)| + \int_0^t |f(\tau)|d\tau$, $t>0$, and $Z(t) := \sup_{\tau\le t}|z_p(\tau)|$, $t\le 0$. This implies
\begin{align*}
  &\cH[Z-|z|](t) = |f(t)| - \sgn(z(t)) f(t) \\
  &+ \int_0^\infty \left( \sgn(z(t)) \phi\left(a,t,\frac{z(t) - z(t-\e a)}{\e}\right) + \phi\left(a,t,\frac{Z(t)-Z(t-\e a) - |z(t)| + |z(t-\e a)|}{\e}\right)\right)da \\
  &\ge \int_0^\infty \left( \sgn(z(t)) \phi\left(a,t,\frac{z(t) - z(t-\e a)}{\e}\right) + \phi\left(a,t,\frac{- |z(t)| + |z(t-\e a)|}{\e}\right)\right)da  \,,
\end{align*}
where we have used the monotonicities of $Z$ and of $\phi$. For $z(t)>0$ we use the oddness of $\phi$ and write the integrand on the right hand side as
$$
  \phi\left(a,t,\frac{z(t) - z(t-\e a)}{\e}\right) - \phi\left(a,t,\frac{z(t) - |z(t-\e a)|}{\e}\right) \ge 0 \,,
$$
by the monontonicity of $\phi$. For $z(t)<0$ the integrand reads
$$
  \phi\left(a,t,\frac{z(t) + |z(t-\e a)|}{\e}\right) - \phi\left(a,t,\frac{z(t) - z(t-\e a)}{\e}\right) \ge 0 \,,
$$
showing $\cH[Z-|z|](t)\ge 0$, $t>0$. Since obviously $Z(t) - |z(t)| \ge 0$ for $t\le 0$, an application of the comparison principle completes the proof.
\end{proof}
}

\section{The regular convex potential} \label{sec.reg}

{In this section the additional assumption $\psi\in C^{1,1}(\R)$ on the potential is used. 
We start with existence results for \eqref{eq.z.nl} and for the formal limit \eqref{eq.zz.nl} as $\e\to 0$.}

%\subsection{Existence}
\begin{theorem}\label{thm.exist.uniq} Let Assumptions \ref{hypo.data} hold and {let furthermore $\psi'$ be Lipschitz on $\rr$.}
Then there exists a unique solution {$\zeps\in C^1(I_T)$} of problem \eqref{eq.z.nl}.
\end{theorem}

\begin{proof}
{Local existence will be proven by Picard iteration as for ODEs in the space $C([0,\tau])$ with $\tau>0$ small enough.
Since this is completely standard, we only prove the contraction property of the fixed point map
$$
    F[z](t) = z_p(0) + \int_0^t v(s)ds - \int_0^t\int_0^\infty \rho(a,s)\psi'\left(\frac{z(s)-z(s-\e a)}{\e}\right)da\,ds \,.
$$
Let $z_1,z_2\in C([0,\tau])$ with $z_1(t)=z_2(t)=z_p(t)$, $t\le 0$. Then we estimate
$$
   \left|F[z_1](t) - F[z_2](t)\right| \le \frac{2L'\tau}{\e} \sup_{s\in(0,T)}\int_0^\infty \rho(a,s)da \sup_{s\in (0,\tau)}|z_1(s)-z_2(s)| \,,
$$
with the Lipschitz constant $L'$ of $\psi'$, showing that $F$ is a contraction for $\tau$ small enough. Existence on $[0,T]$ follows from the
a priori estimate
$$
   |\zeps(t)| \le \sup_{(-\infty,0)}|z_p| + \int_0^t |v(s)|ds \,,
$$
obtained by an application of Lemma \ref{lemm:comp} with $\phi(a,t,u)= \rho(a,t)\psi'(u)$ and $f=v$. Continuous differentiability of $\zeps$ follows
from the continuity of $v$ and $\rho$ with respect to $t$.
}
\end{proof}

\begin{lemm}\label{lem.exist.zz}
	{Let the assumptions of Theorem \ref{thm.exist.uniq} hold.
	Then there exists a unique solution $\zz\in C^1(I_T)$ of \eqref{eq.zz.nl}.}
\end{lemm}

\begin{proof}
This is an initial value problem for an implicit ODE. The monotonicity of $\psi'$ and $\psi'(0)=0$ imply existence and uniqueness of $\dzz(t)$ as well as 
the stability estimate $|\dzz(t)|\le |v(t)|$. By the Lipschitz continuity of $\psi'$ and by $\rho\in C(I_T; L^1(a\,da))$, the left hand side of \eqref{eq.zz.nl} is
continuous as a function of $\dzz(t)$ and $t$. This and the stability estimate imply continuity of $\dzz$, completing the proof.
\end{proof}

\newcommand{\tiuze}{\ti{u}_{0,\e}}

Now we are in the position to prove a convergence result.
\begin{theorem}\label{thm:eps20}
Let the assumptions of Theorem \ref{thm.exist.uniq} hold. Then $\lim_{\e\to 0} \zeps = \zz$ uniformly in bounded subsets of $I_T$.
\end{theorem}
\begin{proof}
A straightforward computation shows that the difference between \eqref{eq.z.nl} and \eqref{eq.zz.nl} can be written as a linearized problem for the error
$\hz := \zeps -\zz$:
$$
	\cle [\hz](t) = \cR_\e(t) \,,\quad t>0 \,;\qquad \hz(t) = 0 \,,\quad t\le 0 \,,
$$
with
\begin{eqnarray*}
	\cle[z](t) &:=& \dot{z}(t) +  \int_\rr k_\e(a,t) \frac{z(t) - z(t-\e a)}{\e}da \,,\\
	k_\e(a,t) &:=& \kernel(a,t)\int_0^1 \psi''\left(s\veps(a,t)  + (1-s)\vz(a,t)\right) ds \,,
\end{eqnarray*}
and with 
\begin{equation}\label{R.eps}
   \cR_\e(t) = \int_0^\infty k_\e(a,t)a\left( \dzz(t) - \frac{\zz(t) - \zz(t-\e a)}{\e a}\right)da \,.
\end{equation}
Since $\psi''\ge 0$, Lemma \ref{lemm:comp} (with $\phi(a,t,u) = k_\e(a,t)u$) can be applied to the linearized problem, giving $|\hz(t)| \le \int_0^t \cR_\e(\tau)d\tau$.
It remains to estimate \eqref{R.eps}. We start with
\begin{eqnarray*}
   \left| \dzz(t) - \frac{\zz(t) - \zz(t-\e a)}{\e a}\right| &\le& \frac{1}{\e a} \int_{t-\e a}^t \left| \dzz(t) - \dzz(s)\right| ds \\
  &\le& \frac{1}{\e a} \int_{t-\e a}^t \left| \dzz(t) - \dzz(s_+)\right| ds + \frac{1}{\e a} \int_{t-\e a}^t \left| \dzz(s_+) - \dzz(s)\right| ds \,.
\end{eqnarray*}
In the first term on the right hand side we use the modulus of continuity $\omega_t$ of $\dzz$ on the interval $[0,t]$. In the second term the integrand is bounded
by Assumption \ref{hypo.data} (ii). Thus,
$$
  \left| \dzz(t) - \frac{\zz(t) - \zz(t-\e a)}{\e a}\right| \le \omega_t(\e a) + \left(|\dzz(0+)| + L_p\right) {\bf 1}_{t-\e a < 0} \,.
$$
Since $k_\e \le L'\rho$ (with the Lipschitz constant $L'$ of $\psi'$ already used above),
\begin{eqnarray*}
  |\cR_\e(t)| &\le& L' \int_0^\infty a\rho(a,t) \omega_t(\e a)da + L' \left(|\dzz(0+)| + L_p\right) \int_{t/\e}^\infty a\rho(a,t)da \\
  &\le& L' \int_0^\infty a\rho(a,t) \omega_t(\e a)da + L' \left(|\dzz(0+)| + L_p\right) \int_0^\infty (1+a)a\rho(a,t)da  \frac{\e}{\e + t}\,.
\end{eqnarray*}
The result follows by integration with respect to $t$ and by using the dominated convergence theorem for the first term on the right hand side.
\end{proof}
\renewcommand{\rhoinf}{\varrho_\infty}
\newcommand{\dzi}{\dot{z}_\infty}
\renewcommand{\hv}{\hat{v}}
The rest of this section is concerned with large time asymptotics. For notational simplicity
the parameter $\varepsilon$ is set to 1, whence \eqref{eq.z.nl} reads
\begin{equation}\label{eq.z1.nl}
 \begin{aligned}
 & \dot z(t) + \int_\rr \psi'\left( z(t)-z(t-a)\right) \kernel(a,t) da= v(t)  \,, & t >0 \,,\\
 & z(t) = \zp(t) \,, &t\le 0 \,.
\end{aligned}
\end{equation}
First we prove that for large time the velocity becomes approximately constant. For the time dependent data, a weak convergence assumption is sufficient, 
in the sense that the difference between the data and its asymptotic limit is integrable up to $t=\infty$.

\begin{theorem}\label{thm:t.inf.vinf.ne.0}
Let the assumptions of Theorem \ref{thm.exist.uniq} hold with $T=\infty$ and let $\vinf\in\R$ and $\rhoinf\in L^1((1+a)da)$ satisfy
$$
  \int_0^\infty |v(t)-\vinf| dt < \infty  \,,\qquad \int_0^\infty \int_0^\infty  a|\kernel(a,t)-\rhoinf(a)| \; da \;dt < \infty \,.
$$
Then there exists a unique solution $\dzi\in\R$ of \eqref{eq.lim.t.large},  such that the solution $z$ of \eqref{eq.z1.nl} satisfies
$$
   z(t) = \dzi t + O(1) \,,\qquad \mbox{as } t\to\infty \,.
$$
\end{theorem}
\newcommand{\hrho}{\hat{\varrho}}
\newcommand{\cl}{{\mathcal L}}
\begin{proof}
Existence and uniqueness for \eqref{eq.lim.t.large} follows as in the proof of Lemma \ref{lem.exist.zz}, and we denote the solution
by $\dzi:= \gamma$. A straightforward computation shows that $\hat z$, defined by $\hat z(t) := z(t) - \dzi t - z_p(0)$, $t>0$, and $\hat z(t) = 0$, $t\le 0$, satisfies the
linearized equation $\cL[\hat z](t) = \cR(t)$, $t>0$, with
\begin{eqnarray*}
   \cL[z](t) &=& \dot z(t) + \int_0^\infty k(a,t)(z(t)-z(t-a)) da \,,\quad k(a,t) = \rho(a,t) \int_0^1 \psi''(su(a,t)+(1-s)a\dzi)ds \,,\\
   \cR(t) &=& v(t)-v_\infty - \int_0^\infty \psi'(a\dzi)(\rho(a,t) - \rho_\infty(a))da \,.
\end{eqnarray*}
Lemma \ref{lemm:comp} can be applied with $\e=1$, $\phi(a,t,u)=k(a,t)u$, and $f = \cR$, to show that, for any $t>0$,
\begin{eqnarray*}
    |z(t) - \dzi t| &\le& |z_p(0)| + |\hat z(t)| \le |z_p(0)| + \int_0^\infty |\cR(\tau)|d\tau \\
    &\le& |z_p(0)| + \int_0^\infty |v(\tau)-v_\infty| d\tau + L'|\dzi| \int_0^\infty \int_0^\infty a|\rho(a,\tau)-\rho_\infty(a)|da\,d\tau \,,
\end{eqnarray*}
completing the proof.
\end{proof}

\renewcommand{\e}{\varepsilon}
\renewcommand{\dz}{\dot{z}}
An improvement of this result, i.e. convergence of $\dot z(t)$ and of $z(t)- \dzi t$, can be achieved under additional assumptions, in particular for vanishing flow velocity
$v$.

\begin{propm}\label{prop.dec.caract}
Let the assumptions of Theorem \ref{thm.exist.uniq} hold with $T=\infty$ and let $v\equiv 0$. Let $\kernel$ satisfy 
$0 \ge (\dt + \da) \kernel\in (L^\infty \cap L^1)(\rr\times\rr)$ and $0\le \kernel(0,t)\in L^\infty(\rr)$. Furthermore let
there exist $\kernel_\infty \in L^1(\rr,(1+a))$ such that $\kernel(\cdot,t) \to \kernel_\infty$ in $L^1(\rr,(1+a))$.
Then the solution of \eqref{eq.z1.nl}
satisfies
{
$$
\int_0^\infty | \dot{z} (t) |^2 dt \leq \int_\rr \kernel (a,0)\psi(z_p(0)-z_p(-a)) da
$$}
and 
$
\lim_{t\to\infty} \dot{z}(t) = 0.
$
\end{propm}

\begin{proof}
Setting $u(a,t):=z(t)-z(t-a)$, the function $\psi(u(a,t))$ solves the transport problem 
$$
(\dt + \da) \psi(u)  = \psi'(u(a,t)) \dot{z},\quad \psi(u(0,t))=0 \quad \text{and}\quad \psi(u(a,0))=\psi(\vepsi(a)) \,.
$$ 
{with $\vepsi(a):=z_p(0)-z_p(-a)$. This connection between the delay equation and age structured population models has already been used
in \cite{MiOel.1}, see also \cite{Diekmann}.} Considering   $\kernel(a,t) \psi(u(a,t))$, it solves in the sense of characteristics (cf  \cite[Theorem 2.1 and Lemma 2.1]{MiOel.1})~:
$$
 (\dt + \da) \kernel  \psi(u)   - ((\dt + \da )\kernel)  \psi(u) =  \kernel \psi'(u(a,t)) \dot{z} ,
$$ 
integrated in age this gives :
$$
\ddt{}  \int_\rr \kernel(a,t)  \psi(u(a,t) )  da \leq   \int_\rr \kernel \psi'(u(a,t)) da \dot{z} = - \dot{z}^2,
$$
which then leads to~:
$$
\left[  \int_\rr \kernel(a,t)  \psi(u(a,t)) da \right]_{s=0}^{s=t} + \int_0^t \dot{z}^2 ds \leq 0.
$$
This shows that $\dot{z}$ belongs to $L^2(\rr)$ since 
$$
\nrm{\dot{z}}{L^2(\rr)}^2 \leq \int_\rr \psi(\vepsi) \kernel(a,0) da < \infty. 
%\nrm{\frac{\vepsi}{1+a}}{L^\infty(\rr)} \nrm{(1+a) \kernel(\cdot,0)}{L^1(\rr)}.
$$
{With the formula $u(a_0,t) = \int_{t-a_0}^t \dot{z} (\tau) d\tau$, $a_0<t$, the Cauchy-Schwarz inequality implies}
$$
| u(a_0,t) | \leq \sqrt{a_0} \nrm{\dot{z}}{L^2(t-a_0,t)}.
$$
Using Lebesgue's Theorem, it is easy to show that $\lim_{t\to \infty} \nrm{\dot{z}}{L^2(t-a_0,t)} =0$.
Thanks to  Lebesgue's Theorem again, one shows that
$$
\int_0^t \rhoinf(a) | u(a,t) | da \to 0
$$
when $t$ grows large. By hypothesis, $\psi'(0)=0$, so that %one has
$$
\left| \int_0^t \psi'(u(a,t)) \rhoinf(a)da \right| \leq \nrm{\psi''}{L^\infty(\RR)} \int_0^t | u(a,t) |\rhoinf(a) da,
$$
which shows that the left hand side also tends to zero as $t$ tends to infinity.

In order to study the convergence of $\int_\rr \kernel(a,t) \psi'(u(a,t)) da$ when $t$ goes to infinity, 
we split the integral in two parts :
$$
\int_\rr \psi'(u) \kernel(a,t) da = \left( \int_0^t + \int_t^\infty \right) \psi'(u) \kernel(a,t) da =: I_1 + I_2.
$$
For the first part one has :
$$
I_1 = \int_0^t \psi'(u) \left(\kernel(a,t) - \rhoinf(a) \right)da + \int_0^t \psi'(u) \rhoinf (a) da
$$
The last term is already estimated above and tends to zero when $t$ goes large.
For the first one, as $\psi'(0)=0$, one has 
$$
%\left\{
\begin{aligned}
\int_0^t & \psi'(u) (\kernel(a,t) -\rhoinf(a)) da %\\
%&
  \leq \nrm{\psi''}{L^\infty} \nrm{\frac{u}{\sqrt{1+a}}}{L^\infty(0,t)} \nrm{(1+a)(\kernel(\cdot,t)-\rhoinf)}{L^1(\rr)}\\
\end{aligned}
%\right.
$$
the latter term vanishing when $t$ grows by hypothesis.
It remains to consider $I_2$. By {the definition of $u$ we have}
$$
u(a,t) = \vepsi(a-t) + \int_0^t \dot{z}(\tau) d\tau \,,\qquad a\ge t \,,
$$
and thus
$$
| u(a,t) | \leq |\vepsi(a-t)| + \sqrt{t} \nrm{\dot{z}}{L^2_t},
$$
which finally provides :
$$
\left| \frac{u(a,t)}{(1+a)} \right| \leq \nrm{\frac{\vepsi}{(1+a)}}{L^\infty_a} + \nrm{\dot{z}}{L^2_t}.
$$
By  Lebesgue's Theorem, this gives that $I_2$ tends to zero as $t$ goes to infinity.
These arguments show that $\dot{z}$ vanishes at infinity since $\dot{z}(t) = - \int_\rr \kernel(a,t) \psi' (u(a,t) )da$.
\end{proof}

{Finally we are able to identify the limit of $z(t)$ under the further assumptions that $\kernel$ is time independent and nonincreasing, and the problem is linear.}
We assume that $\psi(u)=u^2/2$ and {and $\da \kernel(a) \leq 0$, and set $p(a,t) := \int_0^t u(a,\tau ) d\tau= \int_0^t (z(\tau)-z(\tau-a))d\tau$, which solves}
\begin{equation}\label{eq.p}
\left\{
\begin{aligned}
& (\dt + \da ) p = - \int_\rr \kernel(a) p(a,t) da + u_I(a) \,,\quad\text{ a.e. } (a,t) \in (\rr)^2 \\
& p(0,t) =0, \quad
 p(a,0) = 0.
\end{aligned}
\right.
\end{equation}
If $p$ reaches a steady state $p_\infty$, it should satisfy
{$$
    \da  p_\infty(a) = - \int_\rr \kernel(\tia) p_\infty(\tia) d\tia + u_I(a) \,,\qquad p_\infty(0) =0 \,,
$$
with the explicit solution
$$
  p_\infty(a) = \int_0^a \vepsi(\tia)d\tia - a \int_0^\infty \kernel(\tia) \int_0^{\tia} \vepsi(\hat a)d\hat a\, d\tia \left( 1 + \int_0^\infty \kernel(\tia)\tia\,d\tia\right)^{-1} \,.
$$}
\newcommand{\hp}{{\hat{p}}}

Then, setting $\hp(a,t) := p(a,t)-p_\infty(a)$, it solves the homogeneous problem associated with \eqref{eq.p}, 
with the initial condition $\hp(a,0)=-p_\infty(a)$. 
{Multiplication by $\rho\hp$ and integration with respect to $a$ and $t$ gives
$$
  \int_0^\infty \rho \hp^2 da - 2\int_0^t \int_0^\infty \hp^2 \da\rho \,da\,ds + 2\int_0^t \left(\int_0^\infty \rho\hp \,da\right)^2 ds = \int_0^\infty \rho p_\infty^2 da \,.
$$
We use the monotonicity of $\rho$ for the second term and the Cauchy-Schwarz inequality for the first to obtain
$$
   \left(\int_0^\infty \rho\hp \,da\right)^2 + 2 \int_0^\infty \rho\,da \int_0^t \left(\int_0^\infty \rho\hp \,da\right)^2 ds \le \int_0^\infty \rho\,da \int_0^\infty \rho p_\infty^2 da \,,
$$
which implies $\int_0^\infty \rho(a)\hp(a,t)da \to 0$ as $t\to\infty$ using the same arguments as for $\dot{z}$ and $\int_\rr \rho (a) u(a,t) da$ in Proposition \ref{prop.dec.caract}. The simple computation
$$
        z(t) - z_p(0)  + \int_\rr \kernel(a) p_\infty(a) da = -\int_\rr  \kernel(a)\hp(a,t) da 
$$
completes the proof of the following result.}

\begin{propm}
	{
	Let the assumptions of Proposition \ref{prop.dec.caract} hold, let $\kernel$ be independent of $t$ and nonincreasing, and let $\psi(u) = u^2/2$. 
	Then the solution of \eqref{eq.z1.nl} satisfies
	$$
	   \lim_{t\to\infty} z(t) = z_p(0) - \int_0^\infty \rho(a) p_\infty(a)da = \left( z_p(0) + \int_0^\infty \kernel(a) \int_{-a}^0 z_p(\tau)d\tau\right)
	   \left( 1 + \int_0^\infty \kernel(\tia)\tia\,d\tia\right)^{-1} \,.
	$$}
%	and the convergence is exponential.}
\end{propm}

For instance if 
 $\kernel(a):=\beta \exp(- \zeta a)$, where $\zeta$ and $\beta$ are constants, 
{$$
   \lim_{t\to\infty} z(t) = \frac{\zeta^2 z_p(0) + \beta\zeta \int_{-\infty}^0  \exp(\zeta\tau)\zp(\tau) d\tau }{\zeta^2 + \beta} \,.
$$
}

\section{Discontinuous stretching force -- differential inclusions} \label{sec.diff.inclusion}
\renewcommand{\e}{\varepsilon}

{In this section we allow the elastic response function $\psi'$ to be discontinuous. However, different from the preceding
section, we assume its boundedness. Note that in terms of the potential $\psi$ this means that the convexity assumption, which implies local Lipschitz continuity,
is strengthened to global Lipschitz continuity.
We start by smoothing $\psi$, to be able to apply results from the preceding section.}
\newcommand{\psid}{\psi_\delta}
\begin{lemm}\label{lem.reg}
{Let Assumptions \ref{hypo.data} hold and furthermore $\psi\in C^{0,1}(\R)$ with Lipschitz constant $L$. 
	Let $\omega_1$ denote a smooth, even probability density and $\omega_\delta := \delta^{-1}\omega_1(\cdot/\delta)$. 
	Then, for $\delta>0$, $\psi_\delta := \omega_\delta\star\psi - \omega_\delta\star\psi(0)$ is smooth, even, convex, and Lipschitz continuous with Lipschitz constant $L$.
	Furthermore $\psi_\delta'$ is Lipschitz continuous on $\R$ and $\lim_{\delta\to 0+}\psi_\delta = \psi$, uniformly on bounded subsets of $\RR$.}
	\end{lemm}
\begin{proof}
	Since $\psi$ is convex we have
	$$
	\psi(\theta u + (1-\theta) v- y )= \psi(\theta (u - y) + (1-\theta)(v-y)) \leq \theta \psi(u-y) + (1-\theta) \psi(v-y) \,.
	$$
	Integrating against $\omega_\delta(y)dy$ gives the convexity of $\psi_\delta$. {The estimate
	$$
	   |\psi_\delta''(u)| = \left| \int_{\R} \psi'(u-v)\omega_\delta'(v)dv\right| \le \frac{L}{\delta} \int_{\R} |\omega_1'(\eta)|d\eta
	$$
	shows the Lipschitz continuity of $\psi_\delta'$.} The remaining results are
	standard and can be found in basic textbooks (cf. Appendix C Theorem 6 in \cite[Appendix C, Theorem 6]{Evans.Book}).
\end{proof}
\renewcommand{\vepsd}{u_\e^\delta}
\renewcommand{\zepsd}{z_\e^\delta}
\newcommand{\dzepsd}{\dot{z}^\delta_\e}

{\begin{lemm}\label{lem:reg.exist}
Let the assumptions of Lemma \ref{lem.reg} hold. Then problem \eqref{eq.z.nl} with $\psi$ replaced by $\psi_\delta$ has a unique solution
$\zepsd\in C^1(I_T)$, which is, for every compact $\tilde I\subset I_T$, bounded in $C^1(\tilde I)$ uniformly in $\delta$ and $\e$.
\end{lemm}
\begin{proof}
The data with $\psi$ replaced by $\psi_\delta$ satisfy the assumptions of Theorem \ref{thm.exist.uniq}, implying the existence and uniqueness statement.
The obvious estimates
$$
   |\dzepsd(t)| \le \|v\|_{L^\infty(I_T)} + L\left\|\int_0^\infty \kernel(a,\cdot)da \right\|_{L^\infty(I_T)} \,,\qquad
   |\zepsd(t)| \le |z_p(0)| + t \|\dzepsd\|_{L^\infty(I_T)} \,,
$$
complete the proof.
\end{proof}
We shall deal with the lack of smoothness of the potential by passing to a variational formulation analogous to the treatment of gradient flows with nonsmooth
convex potentials (see, e.g., \cite{Evans.Book}). For $t\in I_T$, $z:\, (-\infty,t)\to \R$, and $w\in\R$, we define
$$
  \mathcal{I}_\delta [z,t](w) := \e \int_\rr  \psid\left(\frac{w-z(t-\e a)}{\e}\right) \kernel(a,t) da \,,
$$
which is (for each $\delta>0$) a smooth function of $w$. With the notation from Lemma \ref{lem:reg.exist} we have by the convexity and smoothness 
of $\psi_\delta$ that for each $t\in I_T$
$$
    \zepsd(t) = \argmin_{w\in\R} \left( \mathcal{I}_\delta [\zepsd,t](w) + w(\dzepsd(t) - v(t))\right) \,,
$$
or, equivalently,
\begin{equation}\label{var-inequal-d}
	\mathcal{I}_\delta [\zepsd,t](w) \geq \mathcal{I}_\delta [\zepsd,t](\zepsd(t)) +(v(t)-\dzepsd(t)) (w-\zepsd(t)) \,,\qquad \forall\,w\in\R \,.
\end{equation}
The formal limit 
$$
  \mathcal{I}[z,t](w) := \e \int_\rr  \psi\left(\frac{w-z(t-\e a)}{\e}\right) \kernel(a,t) da \,,
$$
of $\mathcal{I}_\delta$ is still a convex, but not necessarily a smooth function of $w$. 
\newcommand{\subI}{\partial \mathcal{I}}
We define its set valued subdifferential by
$$
\subI [z,t](w) := \left\{ q \in \RR:\, \mathcal{I}[z,t](\hat w) \geq \mathcal{I}[z,t](w) + q (\hat w-w),\quad \forall \hat w \in \RR \right\} \,.
$$
For each $w\in\R$ it is a nonempty closed interval. An existence result, where \eqref{eq.z.nl} is replaced by a differential inclusion can now be 
proven by passing to the limit $\delta\to 0$ in \eqref{var-inequal-d}.
\begin{theorem}\label{thm.z.eps.exist.inclusion}
Let the assumptions of Lemma \ref{lem.reg} hold. Then there exists 
	$\zeps \in C^{0,1}_{loc}(I_T)$ such that, 
for almost every $t\in I_T$, 
 	\begin{equation}\label{eq.z.eps.diff.incl1}
		v(t) -\dzeps(t) \in \partial \mathcal{I}[\zeps,t](\zeps(t)) \,.
	\end{equation}
\end{theorem}
\begin{proof}
By Lemma \ref{lem:reg.exist} and the Arzela-Ascoli theorem, there exists $\zeps \in C^{0,1}_{loc}(I_T)$, such that, as $\delta\to 0$, $\zepsd$ converges 
(up to the choice of an appropriate subsequence) to $\zeps$ uniformly on bounded subintervals of $I_T$. Also $\dzepsd$ converges to $\dzeps$ in 
$L^\infty(I_T)$ weak star, where the notation is justified, since it is equal to the derivative of $\zeps$ almost everywhere in $I_T$. By Lemma 
\ref{lem.reg}, ii) and iii), the integrands
in $\mathcal{I}_\delta [\zepsd,t](w)$ and $\mathcal{I}_\delta [\zepsd,t](\zepsd(t))$ converge pointwise in $a\in\rr$. By the uniform Lipschitz continuities of
$\psid$ and $\zepsd$ the integrands can be bound by $C(1+a)\rho\in L^1(\rr)$. Therefore we can pass to the limit in $\mathcal{I}_\delta [\zepsd,t](w)$ and 
$\mathcal{I}_\delta [\zepsd,t](\zepsd(t))$ by dominated convergence.\\
The last term in \eqref{var-inequal-d} converges in $L^\infty(I_T)$ weak star, as a consequence of the strong convergence of $\zepsd$ and of the weak star
convergence of $\dzepsd$. Therefore the limiting variational inequality
$$
	\mathcal{I} [\zeps,t](w) \geq \mathcal{I}[\zeps,t](\zeps(t)) +(v(t)-\dzeps(t)) (w-\zeps(t)) \,,\qquad \forall\,w\in\R \,,
$$
holds for all Lebesgue points $t$ of $\dzeps$, and this is equivalent to \eqref{eq.z.eps.diff.incl1}.
\end{proof}

The formal limit problem \eqref{eq.zz.nl.lip} is equivalent to
$$
   0 \in \partial \mathcal{J}_t(\dot z_0(t)) \qquad\mbox{with } \mathcal{J}_t(w) = \frac{w^2}{2} - v(t)w + \int_0^\infty \frac{\psi(aw)}{a}\rho(a,t)da \,,
$$
which means that we are looking for a minimizer of $\mathcal{J}_t$. Since this a strictly convex, coercive function, a unique minimizer exists,
showing the existence of a unique solution of \eqref{eq.zz.nl.lip}.

In the following proof we shall need a result on the representation of subdifferentials \cite{clarke.book}.
	With the definition
	$$
	\veps(a,t) := 
	\begin{cases}
	\frac{\zeps(t)-\zeps(t-\e a)}{\e}& \text{ if } t >\e a \\
	\frac{\zeps(t)-\zp(t-\e a)}{\e} & \text{ otherwise} 
	\end{cases},\quad  \vz(a,t) := a \dzz(t) , \quad \text{ for a.e. } (a,t) \in \rr \times (0,T).
	$$
	we define the function
	$$
	f(w) := \int_\rr \psi (\veps(a,t) +w) \kernel(a,t) da \,.
	$$
	As a consequence of $\psi$ being convex and Lipschitz, the subdifferentials of $\psi$ and $f$ coincide with their
	generalized gradients, as defined in \cite[Prop. 2.2.7]{clarke.book}. This allows to use \cite[Theorem 2.7.2]{clarke.book} implying
	$$
	\partial f(w) \subset \int_\rr \partial\psi(w+\veps(a,t)) \kernel(a,t) da \,.
	$$
	As a consequence there exist measurable selections  $\zeta_{\veps}(a,t) \in \partial\psi(\veps(a,t))$ and $\zeta_{\vz}(a,t) \in \partial \psi (\vz(a,t))$  
	such that
	$$
	\dzeps(t) +  \int_\rr \zeta_{\veps} (a,t)\kernel(a,t) da = v(t) \,,\qquad \dzz(t) +  \int_\rr \zeta_{\vz} (a,t)\kernel(a,t) da = v(t) \,.
	$$

\renewcommand{\ovu}{\overline{u}}
\begin{theorem}\label{thm.cvg.pcwz.cuu}
		Assume that $\zeps$ solves the differential inclusion \eqref{eq.diff.inclusion.z.intro}, with 
		\begin{itemize}
			\item $\kernel$ is constant in time, and $\kernel \in L^1 (\rr,(1+a)^2)\cap L^\infty(\rr)$
			\item $v$ is constant,
			\item $\psi$ is convex, $L_\psi$-Lipschitz and there 
			exists a finite set $U:= \{\ovu_{i}, \; i\in \{1,\dots,N\}\}$ such that
			$u_1 < u_2 < \dots < u_N$, 
			$\psi \in C^{1,1}(\rr \setminus U )$	
			and there exists $L_{\psi'}$ such that 
			{$$
			| \psi'(w_1) -\psi'(w_2) | \leq L_{\psi'} | w_2 - w_1 |.
			$$}	
			for all {$(w_1,w_2)$} $ \in (-\infty,u_1)^2 \cup\bigcup_{i=1}^{N-1} (u_i,u_{i+1})^2 \cup (u_N,\infty)^2$.
			%\end{itemize}
		\end{itemize}
		then there exists a unique real $\gamma \in \RR$ solving \eqref{eq.zz.nl.lip}. Moreover if $\gamma\neq0$, then % one has 
		{\begin{equation}
		\label{eq.err.est}
		\nrm{\zeps- \zz}{C([0,T])} \leq C(v,\rho,L_{\psi'},\zp) \e \left| \ln \e \right|
		\end{equation}
		where $\zz(t) = \gamma t + \zp(0)$.}	
\end{theorem}

\begin{proof}
	We prove the result for $N=1$, the general proof for $N>1$ works the same.

	{
		First, if $\gamma$ solves \eqref{eq.zz.nl.lip} with a  kernel $\kernel(a)$ and a source term $v$ both constant in time, 
		then $\gamma$ is constant.
		% Indeed, if there are two different values $\gamma(t_1)\neq \gamma(t_2)$, the difference solves 
		% a homogeneous problem as in \eqref{eq.difference.incl} and thus $\gamma(t_1)= \gamma(t_2)$ which is a contradiction.\\
		For the rest of the proof, we set $\vz (a,t):= a \gamma$ and we assume that $\gamma \neq 0$. 
		Then, one defines $A_{\eta,t} := \{ a \in \rr $ $\st |\veps(a,t)-\vz(a,t)|\leq \eta\}$. % (cf fig. \ref{fig:a_eta_t}). 
		Since, for a fixed $t$, the function of a, $\veps(a,t)-\vz(a,t)$ 
		is continuous, $A_{\eta,t}$ is a closed set. It is also Lebesgue-measurable.
		%Thus if $a \notin A_{\eta,t} $ then necessarily,  $|\veps(a,t)-\vz(a,t)| > \eta$.
		By hypothesis, there exists  $\ovu\in\RR$ such that $\psi \in C^{1,1}(\RR\setminus\{\ovu\})$
		and there exists a constant $L_{\psi'}$ such that 
		$$
		\left| \psi'(u)-\psi'(v)\right| \leq L_{\psi'} |u-v| ,\quad \forall (u,v) \in (-\infty,\ovu)^2 \cup (\ovu,+\infty)^2.
		$$		
		In this context, we consider four cases depending on whether $\gamma>0$ (resp $\gamma <0$) and
		$\ovu \geq 0$ (resp. $\ovu<0$) :
		\begin{enumerate}[i)]
			\item If $\gamma>0$ and $\ovu<0$, 
			we assume that $\eta < |\ovu|$.
			For every $a \in A_{\eta,t}$, %This means that for every $a \in A_{\eta,t}$,
			one has : 
			$$
			\vz(a,t) = \gamma a \geq 0 > \ovu
			$$ 
			and
			$$
			-\eta < \veps(a,t) - \gamma a < \eta
			$$	
			which implies :
			$$
			\ovu < \gamma a + \ovu < \gamma a- \eta < \veps(a,t).
			$$
			This means that for every $a \in A_{\eta,t}$,
			$$
			(\vz(a,t),\veps(a,t)) \in (\ovu,\infty)^2.
			$$
			Both solutions lie in the domain where $\psi'$ is Lipschitz. Thus
			$\zeta_{\veps}(a,t)=\psi'(\veps(a,t))$ and $\zeta_{\vz}(a,t)=\psi'(\vz(a,t))$,
			and thus setting 
			$$
			\cR_\eta(t):=\int_{A_{\eta,t}}\left( \zeta_{\veps}(a,t) - \zeta_{\vz}(a,t)\right) \kernel(a ) da 
			%\leq
			%\eta L_{\psi'} \nrm{\kernel}{L^\infty_t L^1_a}.
			$$
			one has that $|\cR_\eta(t)| \leq \eta L_{\psi'} \nrm{\kernel}{L^1_a}$. 
			The symmetric case when $\gamma <0$ and $\ovu >0$ works the same provided again that $\eta < \ovu$.
			\item If instead, $\gamma >0$ and $\ovu \geq 0$, there exists $a_0 \geq 0$ such that $\ovu = \gamma a_0$.  
			We split the previous integral in two parts~:
			$$
			\cR_{\eta}(t) = \left(\int_{A_{\eta,t}\cap B(a_0,\omega)}+\int_{A_{\eta,t} \setminus B(a_0,\omega)} \right) \left( \zeta_{\veps}(a,t) - \zeta_{\vz}(a,t)\right)  \kernel(a)da =: I_1(t) + I_2(t)
			$$
			where $\omega$ is a small  positive parameter yet to be fixed.\\
			The first term can be bounded by the measure of $B(a_0,\omega)$, indeed :
			\begin{equation}\label{eq.I1}
			|I_1(t)| \leq 2 L_\psi \int_{B(a_0,\omega)} \kernel(a) da  \leq C \omega
			\end{equation}
			the latter bound being possible since $\kernel$ is also a bounded function. \\
			Next, if $a \in A_{\eta,t} \setminus B(a_0,\omega)$ we start by choosing
			$a \leq a_0 - \omega$. Moreover, we assume that 
			\begin{equation}\label{eq.cond.omega}
			\boxed{\omega > \frac{\eta}{\gamma}}
			\end{equation}
			These two latter inequalities allow to write :
			%		First, we simply write :
			$$
			0 < \eta < \omega \gamma \leq \gamma (a_0-a) = \vz(a_0,t)-\vz(a,t) = \ovu - \vz(a,t)
			$$
			which implies obviously that  $\vz(a,t) < \ovu-\eta < \ovu$. Since $a\in A_{\eta,t}$, 
			$$
			\veps(a,t) < \eta + \vz(a,t) < \eta + \vz(a_0,t) - \eta = \ovu
			$$
			so that $\veps(a,t) < \ovu$ as well. 
			This implies that : for $a \in A_{\eta,t}$ and $a\leq a_0-\omega$, 
			$(\veps(a,t),\vz(a,t)) \in (-\infty,-\ovu)^2$. \\
			If $a\geq a_0+\omega$ and $a \in A_{\eta,t}$, then 
			one shows in the same way that : $(\veps(a,t),\vz(a,t)) \in (\ovu,\infty)^2$. \\
			The case when $\gamma <0$ and $\ovu \leq 0$ follows exactly the same lines and leads to the same
			conclusion : when $a\in A_{\eta,t}\setminus B(a_0,\omega)$, provided that \eqref{eq.cond.omega} holds :
			$$
			(\veps(a,t),\vz(a,t)) \in (-\infty,\ovu)^2 \cup (\ovu,\infty)^2.
			$$
			Thus $\zeta_{\veps}(a,t) = \psi'(\veps(a,t))$
			and $\zeta_{\vz}(a,t) = \psi'(\vz(a,t))$ and again
			$$
			\forall a \in A_{\eta,t} \setminus B(a_0,\omega), \; \; 	| \zeta_{\veps}(a,t) - \zeta_{\vz}(a,t) | \leq L_{\psi'}  \eta
			$$
			which shows that 
			\begin{equation}\label{eq.I2}
			\left|I_2(t)\right| \leq L_{\psi'} \eta \nrm{\kernel}{L^1_a}
			\end{equation}
			So, if for instance $\omega = 2 \eta / |\gamma|$, combining \eqref{eq.I1} and \eqref{eq.I2}, we have proved that :
			$$
			\left|\cR_\eta(t) \right| \leq \frac{C \eta }{|\gamma|}.
			$$
			One shall remark firstly that $\eta$ can be made arbitrarily small and 
			that the latter bound is uniform with respect to $\e$.
		\end{enumerate}
		Setting again $\hz(t) := \zeps(t) - \zz(t)$, we shall write the difference equation solved by $\hz$~:
		$$
		\dt \hz + \int_{\rr}\left( {\zeta_{\veps(a,t)} - \zeta_{\vz(a,t)}}\right) \kernel(a) da = 0.
		$$
		We rewrite the last integral term on the left hand side as 
		$$
		\begin{aligned}
		\int_{\rr}  \left(\zeta_{\veps(a,t)} - \zeta_{\vz(a,t)}\right) \kernel(a) da
		=  &
		\int_{\rr \setminus A_{\eta,t}} \frac{\zeta_{\veps(a,t)} - \zeta_{\vz(a,t)}}{\veps(a,t)-\vz(a,t)} 
		(\veps(a,t)-\vz(a,t))\kernel(a) da \\
		&	+  \int_{ A_{\eta,t} } {\zeta_{\veps(a,t)} - \zeta_{\vz(a,t)}} 
		\kernel(a) da
		\end{aligned}
		$$
		that becomes :
		$$
		\dt \hz + \int_{\rr \setminus A_{\eta,t}} \frac{\zeta_{\veps(a,t)} - \zeta_{\vz(a,t)}}{\veps(a,t)-\vz(a,t)} 
		(\veps(a,t)-\vz(a,t))\kernel(a) da = - \cR_\eta,
		$$
		and we denote 
		\begin{equation}
			\label{eq.def.k.eps}
			k_\e(a,t) := \frac{\zeta_{\veps(a,t)} - \zeta_{\vz(a,t)}}{\veps(a,t)-\vz(a,t)} \kernel(a) \chiu{\rr \setminus A_{\e,\eta, t}}{(a)}.
		\end{equation}
		Since the subdifferential of $\psi$ is monotone,  $k_\e$ is  positive, moreover it is a function in $L^1(\rr,(1+a)^2)$. Indeed 
		\begin{equation}\label{eq.borne.keps}
		0 \leq k_\e(a,t) \leq {2  L_\psi} \kernel(a)/{\eta}. 
		\end{equation}
		Our problem can thus be rephrased as
		%	Now%, as in the proof of Proposition \ref{prop.cle}, 
		%	 we extend $\zz(t)=\zp(0)$ for all $t \leq 0$, and 
		%	one writes :
		\begin{equation}\label{eq.dz.psip.disc}
		\dt \hz + \int_{\rr} k_\e (a,t) \left\{ \veps(a,t) - \vz(a,t) \right\} da = -\cR_\eta,
		\end{equation}
		that becomes : 
		$$
		\dt \hz + \int_{\rr} k_\e (a,t) \left\{ \veps(a,t) - \tiuze(a,t) \right\} da = 
		-\int_\rr  k_\e(a,t) (\tiuze(a,t) - \vz(a,t) ) da 
		-\cR_\eta,
		$$
		where
		$$
		\tiuze (a,t) := 
		\begin{cases}
		\frac{\zz(t)-\zz(t-\e a)}{\e} = \gamma a & \text{if} \; t\geq \e a \\
		\frac{\zz(t)-\zp(0)}{\e}  = \frac{\gamma t}{\e} & \text{ otherwise}.
		\end{cases}
		$$
		Thanks to this latter definition the first term in the right hand side above can be reduced to
		$$
		\int_\rr  k_\e(a,t) (\tiuze(a,t) - \vz(a,t) ) da = \ue \int_\tse^\infty \left(\tse -a \right) k_\e(a,t) da
		$$
		Then we rewrite \eqref{eq.dz.psip.disc} as :
		\begin{equation}\label{eq.hat.zz.lip}
		\begin{aligned}
		\cT_\e [\hz](t) & =  
		\ue \int_{\tse}^{+\infty} k_\e(a,t) \left(a -\tse + \hz(t-\e a)\right) da  - \cR_\eta,
		\end{aligned}
		\end{equation}
		where $\cT_\e$ is defined as
		$$
		\cT_\e [\hz](t)  := \dt \hz(t) + \ue \left(\int_{\rr} k_\e(a,t) da\right)  \hz(t) - \ue \int_0^\tse k_\e(a,t) \hz(t-\e a) da .
		$$
		The first term in the right hand side of \eqref{eq.hat.zz.lip} can be estimated
		%as in the proof of Theorem \eqref{thm:cvg.nl} 
		thanks to \eqref{eq.borne.keps}~:
		\begin{equation}\label{eq.tail.estimates}
		\left| \ue \int_{\tse}^{+\infty} k_\e(a,t) \hz(t-\e a) da\right| \leq \frac{4 L_\psi (1+L_{\zp}) \nrm{(1+a)^2\kernel}{L^1_a}}{\eta (1+\tse)}.
		\end{equation}
		At this step, we have proved that
		$$
		\cT_\e [\hz ](t) \leq m(t) := C \left( \frac{\eta}{|\gamma|} + \frac{1}{\eta (1+t/\e)}\right).
		$$
		An easy computation shows that
		$$
		\cT_\e \left[\left|\hz\right|\right] (t) \leq  \sgn(\hz(t)) \cT_\e [\hz ](t) \leq m(t)
		$$
		and since $\int_0^t m(s)ds$ is non-decreasing and non-negative, one has
		$$
		\cT_\e \left[ \int_0^t m(\tau) d\tau \right] \geq m(t)
		$$
		leading to the inequality :
		$$
		\cT_\e \left[\left|\hz\right|\right] (t) \leq \cT_\e \left[ \int_0^t m(\tau) d\tau \right] 
		$$
		We are in the framework of \cite[Generalized Gronwall Lemma 3.10, p. 298]{Gripenberg_ea} and we write :
		%We use 
		%The same comparison principle applies as in Theorem \ref{thm:cvg.nl}, and we write :
		$$
		| \hz(t) | \leq \int_0^t m(\tau) d\tau 
		%=\frac{C \e \ln|\e|}{\eta } + \left| \int_0^t \cR_\eta ds \right| \leq 
		= C \left(\frac{ \e \ln|\e|}{\eta } + t  \frac{\eta}{|\gamma|}\right).
		$$
		Then, setting $\eta = \sqrt{\e\ln|\e|}$,  one obtains the error estimates \eqref{eq.err.est} which ends the proof.}
\end{proof}

{\begin{theorem}\label{thm.cvg.eps.nl.sing}
		Let $\zeps$ solve the differential inclusion \eqref{eq.diff.inclusion.z.intro}, with 
		\begin{compactenum}[i)]			
				\item The kernels $\kernel_\e$ and $\kernel_0$ are such that : 
				\begin{itemize}
					\item $\kernel_\e \in L^1\cap L^\infty(\rr\times(0,T))$
					\item $\kernel_0 \in L^1(\rr,(1+a)^2)\cap L^\infty(\rr)$ is constant in time.
				\end{itemize}
				with $\kernel_\e - \kernel_\infty$ tending to zero 
				in  $L^1(\rr\times(0,T))$.
				\item the source term $v_\e \in W^{1,\infty}(0,T)$ and $v_0 \in \RR$ such that 
				$v_\e \to v_0 \in \RR^*$  in $L^1(0,T)$,
			\item $\psi$ satisfies hypotheses of Theorem \ref{thm.cvg.pcwz.cuu}, 
		\end{compactenum}
		the same conclusions as in Theorem \ref{thm.cvg.pcwz.cuu} hold.
	\end{theorem}}
{\begin{proof}
		As this is an minor extension of Theorem \ref{thm.cvg.pcwz.cuu} we  only
		point out the necessary extra arguments.
		The difference $\hz$ satisfies now :
		$$
		\dt {\hz} + \int_\rr (\zeta_{\veps}-\zeta_{\vz}) \rhoz(a) da = \int_\rr \zeta_{\veps}(\rhoz-\rhoe) da + v_\e(t) -v_0
		$$		
		which following the same arguments as above becomes :
		$$
		\dt {\hz} + \int_\rr (\veps(a,t)-\vz(a,t))k_\e(a,t) da = - \cR_\eta +
		\int_\rr \zeta_{\veps}(\rhoz-\rhoe) da + v_\e(t) -v_0
		$$
		where $k_\e$ is defined in \eqref{eq.def.k.eps}. Since one obtains as above :
		$$
		\cT_\e [|\hz| ](t) \leq m(t) := C \left( \frac{\eta}{|\gamma|} + \frac{1}{\eta (1+t/\e)}
		+ L_\psi \int_\rr |\rhoe(a,t)-\rhoz(a)| da + |v_\e(t) -v_0|\right).
		$$
		The same comparison principle as in Theorem \ref{thm.cvg.pcwz.cuu}, then provides the claim
		integrating $m$ in time.
\end{proof}}

%}{Why this one is suppressed ?}

\begin{rmkm}
	If $\psi$ is only Lipschitz and convex, then its derivative has at most a
	countable set of points in $\RR$ where it is discontinuous. Hypotheses above on $\psi$ assume a finite number of isolated jumps of $\psi'$ on the real line.
	To our knowledge it is not possible to extend the previous proof to this
	general case.	
	Nevertheless, 
	for practical applications (cf, for instance, examples in \cite{pmid29571710} and Section \ref{sec.abs.val}) it seems sufficient.
%	A more interesting problem, 
\end{rmkm}

Here we present a new way to recover large time asymptotics thanks to the $\e$ scaling above.
\begin{theorem}\label{thm.t.large.Lipschitz}
	Under Assumptions \ref{hypo.data}, and assuming that 
	\begin{compactenum}[1)]
		\item  $v_\infty \in \RR$ {
and $v \in W^{1,\infty}	(\rr)$ is such that
$$
\int_\rr \left| v(t)-v_\infty\right| dt < \infty.
$$
}% when $t$ goes to infinity, % \infty$
		\item $\kernel_\infty\in L^1(\rr,(1+a))$ such that {
		$$
		\int_\rr \int_\rr \left| \kernel(a,t)-\kernel_\infty(a)\right| da dt < \infty.
		$$
	}
		\item if $\psi$ satisfies assumptions of Theorem \ref{thm.cvg.pcwz.cuu}, 
	\end{compactenum}
	then  if $z$ solves
	\begin{equation}\label{eq.z.inclusion.eps.one}
		\begin{aligned}
		(v(t)-\dot{z}(t))(w-&z(t)) + \e \int_\rr \psi\left({z(t)-z(t-\e a)}\right)  \kernel(a,t) da   \\
		& \leq    \int_\rr \psi\left({w-\zeps(t-\e a)}\right) \kernel(a,t) da, \quad \forall w \in \RR.
		\end{aligned}
	\end{equation}
	when $t$ goes to infinity, 	
	 there exists $\zz(\tit):=  \gamma t$  such that 
	\begin{equation}\label{eq.asymptotic.lim.time}
	\lim_{t \to \infty} \left| \frac{z(t)}{t}-\zz(1) \right| =0
	\end{equation}
	where $\gamma$ solves \eqref{eq.zz.nl.lip.cst}
	% \begin{equation}\label{eq.diff.inculsion.limit.t}
	% \begin{aligned}
	%  (v_\infty - \gamma) & w + \int_\rr \psi( a \gamma )\kernel_\infty (a) da \\
	%&\quad \quad\leq  \int_\rr  \psi( w+ a \gamma) \kernel_\infty (a) da, \quad \forall w \in \RR.
	% \end{aligned}
	% \end{equation} 
\end{theorem}

\begin{proof}
		We consider the solution $z$ of the problem \eqref{eq.z.nl} on the time interval $(0,1/\e)$,
		where $\e>0$ is an arbitrarily small parameter.
		We set $\zeps(\tit):=\e z(\tit/\e)$ and $\zpeps(\tit) := \e \zp(\tit/\e)$, then one has :
		\begin{equation}\label{eq.chg.var}
		%\left\{
		\begin{aligned}
		\dtit \zeps (\tit) & = \dt z(\tit/ \e), \\
		\veps(a,\tit)&  := \frac{\zeps(\tit)-\zeps(\tit-\e a)}{\e}=  z(\tit/\e )-z(\tit/\e -a)=:u(a,\tit/\e)
		\end{aligned}
		%\right.
		\end{equation}
		%where $\e$ is an arbitrary small constant. 
		So, if $z$ solves \eqref{eq.z.inclusion.eps.one}, then
		$\zeps$ solves \eqref{eq.diff.inclusion.z.intro}.
		By Theorem \ref{thm.cvg.eps.nl.sing}, 
		$\zeps(\tit)$ converges to $\zz(\tit):=\int_0^{\tit} \gamma(\tau) d\tau$ 
		in $C([0,1])$. This gives for instance that
		$$
		\lim_{\e\to 0} | \zeps(1) - \zz(1) | =0.
		$$
		One then returns to $z$ thanks to the change of unknowns and 
		setting $t=1/\e$ implies \eqref{eq.asymptotic.lim.time} 
		which completes the claim.
	\end{proof}

% \begin{theorem}
% 	Under the previous hypotheses, $\zepsd$ solving \eqref{eq.z.eps.diff.incl} converges strongly in $C([0,T])$ to 
% 	$\zz$ solution of \eqref{eq.zz.lip} when $\e$ goes to zero.
% \end{theorem}
% \begin{proof}
% 	For every $\eta$ we fix $\delta$ such that
% 	$$
% 	\nrm{\zzd - \zz}{C([0,T])} \leq \frac{\eta}{3}
% 	$$
% 	then thanks to Threorem \ref{thm:cvg.nl} for $\delta$ fixed there exists $\e$ small enough such that
% 	$$
% 	\nrm{\zepsd-\zzd}{C([0,T])} \leq \frac{\eta}{3}
% 	$$
% 	which ends the proof.
% \end{proof}
%\subsection{Asymptotics when $t$ grows large and $\e=1$}
%\subsubsection{The general case}
%The existence part is done for a fixed $\e$, it remains to detail changes necessary to adapt the
%previous result to this context.
%Let $z$ be the solution of the differential inclusion :
%\begin{equation}\label{eq.diff.inclusion.z}
%\begin{aligned}
% \forall w \in \RR,  \quad &  (v(t) - \dot{z}) (w-z(t)) \; + \\
% + \int_\rr & \kernel(a,t) \psi(z(t) -z(t-a)) da 
%   \leq  \int_\rr \kernel(a,t) \psi(w -z(t-a)) da,
%\end{aligned}
%\end{equation}
%for a.e. $t\in \rr$, while $z(t)=\zp(t)$ when $t\leq 0$.

%\subsection{The linkages' density is dissipative}
%
%Here we restrict ourselves to the case when $\kernel$ is constant
%in time and solves 
%\begin{equation}\label{}
%\left\{
%\begin{aligned}
%& \da \kernel + \zeta (a) \kernel =0 & \alev a \in \rr \\
%& \kernel(0) = \beta
%\end{aligned}
%\right.
%\end{equation}
%which can be solved explicitly. 
%
%First we establish an {\em a priori} estimate for 
%the regularized equation \eqref{eq.veps.nl} with 
%

\section{An example from the literature}\label{sec.abs.val}
Here we consider the elastic response $\psi(u)=|u|$. 
In a first step assuming that the data $(\kernel,v)$
are constant in time, we study the asymptotic limit
\eqref{eq.zz.nl.lip.cst} and solve it explicitly (cf section \ref{sec.asymptotic}). 

Then  
assuming  a specific form of linkages' distribution
we do not account for any past positions at time $t=0$. 
We show, in this framework,  that it is possible to 
solve explicitly \eqref{eq.zz.nl.lip}  in section \ref{sec.exact} and we illustrate 
numerically this fact in the last part.
\subsection{Study of the limit equation \eqref{eq.zz.nl.lip.cst}}\label{sec.asymptotic}
\begin{propm}
	We suppose that the kernel $\kernel$ is non-negative and satisfies 
	$\kernel(a,t)=\rhoinf(a) \in L^1(\rr)$.
	Assume that $\gamma(t)$ solves \eqref{eq.zz.nl.lip} then it is constant and
	\begin{enumerate}[i)]
		\item if $\gamma>0$  then $v_\infty = \gamma + \mu_\infty$,%$\gamma(t)=\vinf-\mu_\infty$,
		\item  if $\gamma<0$  then $v_\infty = \gamma - \mu_\infty$, %$\gamma(t)=\vinf+\mu_\infty$,
		\item if $\gamma=0$  then $\vinf \in [-\mu_\infty,\mu_\infty]$,
		%\item the solution is unique.
		\item If $\vinf \in [-\mu_\infty,\mu_\infty]$ then $\gamma=0$
	\end{enumerate}
\end{propm}

\begin{proof}
	As in the proof of Theorem \ref{thm.cvg.pcwz.cuu}, if $\gamma$ solves \eqref{eq.zz.nl.lip} 
	with constant data, it is constant.

	In the first case, if $\gamma>0$, then choosing $w<0$ implies that
	$$
	\begin{aligned}
	w (\vinf-\gamma) & + \gamma \int_\rr a \rhoinf da   \\
	  & \leq \gamma \left( \int_{-\frac{w}{\gamma}}^\infty a \rhoinf da - \int_0^{-\frac{w}{\gamma}} a \rhoinf da \right)
	 % \quad \quad  
	 + w  \left( \int_{-\frac{w}{\gamma}}^\infty  \rhoinf da - \int_0^{-\frac{w}{\gamma}}  \rhoinf da \right) .
	\end{aligned}
	$$
	Using Lebesgue's Theorem and taking the limit when $w$ goes to $0^-$ gives
	that $\vinf - \mu_\infty \geq \gamma > 0$.
	In a same way, if $\gamma <0$, expressing \eqref{eq.zz.nl.lip} for positive values of $w$ and taking the limit when $w\to 0^+$ provides that 
	$v+\mu_\infty \leq \gamma < 0$.
	
	On the other hand if $\gamma >0$ (resp. $\gamma <0$) then choosing $w>0$
	(resp. $w<0$) gives straightforwardly that $\vinf - \mu_\infty \leq \gamma$
	(resp. $\vinf + \mu_\infty \geq \gamma$), which concludes the proof of  i) and ii).
	Taking $\gamma=0$ in \eqref{eq.zz.nl.lip} provides that
	$$
	\vinf w \leq \mu_\infty |w|
	$$
	which ends the third claim.
	
	For the last part, if there exists two distinct non-zero solutions $\gamma_i$ for $i\in\{1,2\} $,
	if they have the same sign, they are equal since then i) or ii) hold.
	If their signs are opposite then we end up with a contradiction since then
	$\vinf - \mu_\infty >0$ and $\vinf + \mu_\infty <0$ at the same time.
	Remains  the case when one of the two solution only is zero (for instance $\gamma_1=0$). 
	In this case again we have a contradiction since then 
	$
	\vinf \notin [-\mu_\infty; \mu_\infty] 
	$ (since $\gamma_2 \neq 0$) 
	and $\vinf \in [-\mu_\infty; \mu_\infty] $.
	
	If $\vinf \in (-\mu_\infty,\mu_\infty)$, then $\gamma=0$ is a solution of \eqref{eq.zz.nl.lip} since
	$$
	\vinf w \leq \mu_\infty |w|, \quad \forall w \in \RR
	$$
	which is \eqref{eq.zz.nl.lip} for $\gamma=0$. By uniqueness, it is the only one.
\end{proof}

In fig. \ref{fig.gamma}, we plot the solution $\gamma$ of \eqref{eq.zz.nl.lip} in
the case when $\kernel(a,t)=\rhoinf(a)$ and $v=\vinf$.
\begin{figure}[ht!]
	\begin{center}
		\begin{tikzpicture}
		\begin{axis}[my style, minor tick num=1]
		\addplot[domain=-3:0,red,line width=1.0mm]{x-1};
		\addplot[domain=0:3,red,line width=1.0mm]{x+1};
		\addplot[red,line width=1.0mm] coordinates {(0,-1) (0,1)};
		%\draw[red,thick,dashed] (1,-1)--(0,1);
		%
		%                \addplot[red,mark=*] coordinates {(0,0.45)} node[pin=0:{$(0,r)$}]{} ;
%		\addplot[red,mark=*] coordinates {(0,0.45)} node[left]{$(0,r)$} ;
		\addplot[domain=-5:5, blue]{0.75};
		\addplot[blue] coordinates  {(3,0.8)} node[above] {$\vinf$};
		%\addplot[green,mark=*] coordinates {(3,0.8)} node[{$\vinf$}]{} ;
		\end{axis}
		%\draw[red,line width=1.0mm] (0,-1)--(0,1);
		\end{tikzpicture}
	\end{center}
	\caption{The velocity-force diagram when $\psi(u)=|u|$ and $\int_\rr \rhoinf (a)da =1$}\label{fig.gamma}
\end{figure}
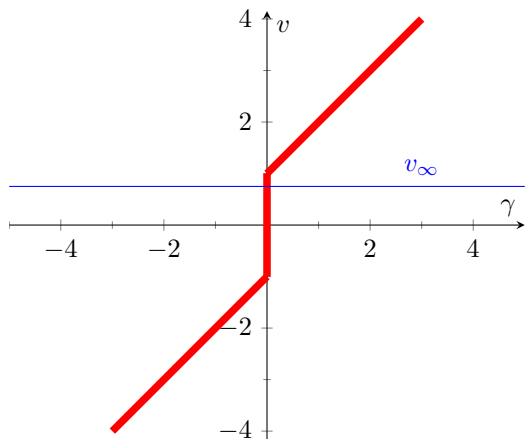

%\begin{theorem}\label{thm.model.3}
%	Under the assumptions of Theorem \ref{thm.t.large.Lipschitz}, and if  $\psi(u)=|u|$, then 
%$$
%\lim_{t \to \infty} \left| \frac{z(t)}{t}- \gamma \right|=0
%$$
%where $\gamma$ is defined as in \eqref{eq.def.gamma.sol}.
%\end{theorem}
%One remarks that if $\vinf \in [-\mu_\infty,\mu_\infty]$, this limit does not characterize enough the 
%behavior of $z$ at infinity. Indeed, one expects that $z$ tends to some constant
%or remains bounded, which this result does not show. It seems, as studied 
%in the regular case in Section \ref{subsec:vinf.eq.zero}, that, in general, this 
%is challenging open question. We show how one can answer in a biologically sound
%but simplified context in the next section.
%
%\subsection{No initial population of linkages and $v=\vinf \in (0,\mu_\infty)$}
\newcommand{\dotz}{\dot{z}}
\subsection{The exact solution of \eqref{eq.diff.inclusion.z.intro}} \label{sec.exact}
We assume here in \eqref{eq.diff.inclusion.z.intro} that the kernel is such that
$\kernel(a,t) = \rhoinf(a) \chiu{\{a<t\}}(a,t)$.
Thus, we solve the problem : find $z \in \Lip (\rr)$ solving
\begin{equation}\label{eq.sec.333}
	(\vinf - \dotz(t)) w + \int_0^t \rhoinf(a) | u(a,t) | da \leq 
	\int_0^t \rhoinf(a) | u(a,t) + w| da, \quad \forall t > 0,
\end{equation}
 together with the initial condition $z(0)=z^0$.
 
 \begin{theorem}\label{thm.plastic}
 	Assume that $\rhoinf$ is a positive monotone non-increasing function in $L^1(\rr)$.
 	We set $\mu_\infty(t) = \int_0^t \rhoinf(a) da$ that tends to $\mu_\infty$ when $t$ goes to infinity.
 	Let's assume moreover that $\vinf \in[-\mu_\infty, \mu_\infty]$  then 
 	the only solution of \eqref{eq.sec.333} is 
 	\begin{equation}\label{eq.zinf}
 		z(t) = \begin{cases}
 		z^0+ \int_0^t[\vinf-\mu_\infty(\tau)]_+ d\tau, & \text{if} \; \vinf \geq 0,\\
 		z^0+ \int_0^t[\vinf+\mu_\infty(\tau)]_- d\tau, & \text{if} \; \vinf \leq 0. 
 		\end{cases}
 	\end{equation}
 	which tends as $t$ grows large to $\zinf = z(t_1)$ where $t_1$ is such that $\mu_\infty(t_1)=\vinf$.
 \end{theorem}
 
\begin{proof} 
We assume hereafter that $\mu_\infty > \vinf \geq 0$, since the opposite case  works the same.	
 A simple computation gives that
 $$
 | \vinf - \dotz | \leq \int_0^t \rhoinf (a )da =: \mu_\infty(t),
 $$
 which shows that 
 %in a neighborhood $[0,t_0]$ of $t=0$, the time-derivative
 % $\dotz >0$, since $\vinf >0$ and $\mu_\infty(t)$ is arbitrarily small when 
 % $t$ is small. One deduces also that
  $0< \vinf -\mu_\infty(t) \leq \dz(t)$ on $[0,t_1)$, where $t_1$ 
  is the  time  for which $\mu_\infty(t_1) = \vinf$.

 In this case setting $u(a,t):= \int_{t-a}^t \dotz(\tau) d\tau$, shows that
 $u(a,t) \geq 0$, for $(a,t) \in \{ (a,t) \in [0,t_1]^2$ such that $ \; a\leq t\}=:\Gamma(t_1)$. For $t$ fixed one has that $u(a,t)$ is increasing with respect to $a \in [0,t]$ and absolutely continuous. Thus 
 there exists $a(w) \in [0,t]$ such that $u(a,t) \leq w$ for all $a \in [0,a_0(w)]$
 and $u(a,t)\geq w$ for $a\in [a_0(w),t]$, this gives
 $$
 (\vinf-\dz(t),-w) \leq w \left( \int_0^{a_0(w)} \rhoinf (a)da - \int_{a_0(w)}^{t} \rhoinf(a)da \right)
 - 2 \int_0^{a_0(w)} \rhoinf(a)u(a,t) da, 
 $$
  for all $w \in [0,u(t,t)]$, then passing to the limit wrt $w\to 0$ gives thanks to the
  integrability of $\rhoinf(a) u(a,t)$ close to $a=0$, and since $a_0(w) \to 0$
  when $w\to 0$, that  : $\dz(t) \leq \vinf - \mu_\infty(t)$.
 So on $[0,t_1]$, 
 \begin{equation}\label{eq.dz}
 	\dz(t) = \vinf-\mu_\infty(t)
 \end{equation}
 Thus $u(a,t)= \int_{t-a}^{t} \vinf - \mu_\infty(\tau) d \tau$ for every $(a,t)\in \Gamma(t_1)$.

We assume that on $(t_1,t_1+\delta)$, with $\delta$ a small positive parameter, $\dz$ is negative definite. We fix 
$t\in (t_1,t_1+\delta)$. As $z$ is monotone increasing
on $(0,t_1)$, there exists $\tau_1$ such that for all $\tau \leq \tau_1$,
$z(\tau) \leq z(t)$, while for  $\tau \in (\tau_1,t)$, $z(\tau )\geq z(t)$.
We set $\eta >0$ a small parameter such that $t-\eta$ still belongs to $(t_1,t_1+\delta)$,
there exists $\tau_2$ depending on $\eta$ such that $z(\tau)$ is in $(z(t-\eta),z(t_1))$
for $\tau \in (\tau_2,t-\eta)$, while $z(t-\eta) > z(\tau) $ for 
$\tau$ in $(0,\tau_2) \cup (t-\eta,t)$ (see fig. \ref{fig.tau}).

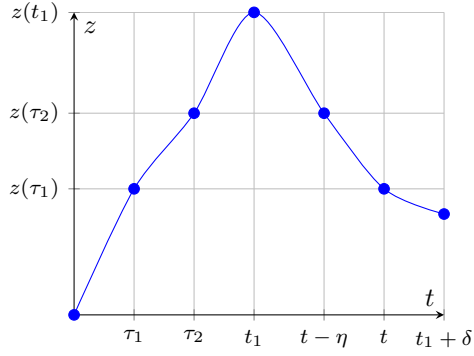
\begin{figure}
	\begin{center}
\begin{tikzpicture}
\begin{axis}[small,axis x line=middle, axis y line= middle, xlabel={$t$}, ylabel={$z$},
xtick={0,3,6,9,12.5,15.5,18.5},
xticklabels={%
	0,$\tau_1$,$\tau_2$,$t_1$,$t-\eta$,$t$, $t_1+\delta$},
ytick={0,2.5,4,6},	
yticklabels={
	0,$z(\tau_1)$,$z(\tau_2)$,$z(t_1)$%,$z(t)$,$z(t_1+\delta)$
},
grid=major
%  note: \frac can be done automatically:%  
%xticklabel style={/pgf/number format/frac},
]
\addplot[smooth,blue,mark=*]coordinates 	{	
	(0, 	0 ) 	
	(	3 ,  2.5)   
	(	6 ,	4 )    
	(	9  , 6)	  
	(12.5 ,4  )   
	(15.5,	2.5)
	(18.5, 2)}	;
\end{axis}
\end{tikzpicture}
\end{center}
\caption{When we assume that $\dz(t)<0$ on $(t_1,t_1+\delta)$}\label{fig.tau}
\end{figure}
One recovers from \eqref{eq.sec.333}, that
$$
\begin{aligned}
(\vinf-\dz(t)) & (z(t-\eta)-z(t)) \\
& + \underbrace{\int_0^{\tau_1} (z(t)-z(\tau) ) \rhoinf(t-\tau) d\tau
	+ \int_{\tau_1}^{t} (z(\tau) -z(t)) \rhoinf(t-\tau) d\tau}_{I_1} \\
&\quad \leq 
\underbrace{\int_0^{t}|(z(t-\eta)-z(\tau) | \rhoinf(t-\tau) d\tau}_{I_2}.
\end{aligned}
$$
We analyze the terms $I_1$ and $I_2$ :
$$
I_1 = z(t)  \left( \int_0^{\tau_1} - \int_{\tau_1}^t  \right)  \rhoinf(t-\tau) d\tau + \left( \int_{\tau_1}^t-\int_0^{\tau_1}   \right) z(\tau) \rhoinf(t-\tau) d\tau ,
%\left( \int_0^{t-\tau_1} - \int_{t-\tau_1}^t \right) z(t-a) \rhoinf(a) da
$$
while 
$$
\begin{aligned}
I_2 & =  z(t-\eta )\left(\int_{t-\eta}^{t} + \int_{0}^{\tau_2} - \int_{\tau_2}^{t-\eta}\right) \rhoinf(t-\tau) d \tau \\
& - \left(\int_{t-\eta}^{t} + \int_{0}^{\tau_2} - \int_{\tau_2}^{t-\eta}\right) z(\tau )\rhoinf(t-\tau) d \tau.
\end{aligned}
$$
This leads to write :
$$
\begin{aligned}
& (\vinf-\dz(t))  (z(t-\eta)-z(t)) + 
(z(t) - z(t-\eta))\left\{ \left( \int_{0}^{\tau_1} - \int_{ \tau_1}^{t}\right) \rhoinf(t-\tau) d\tau   \right\}\\
& + 2  \left( \int_{\tau_1}^{\tau_2} + \int_{t-\eta}^{t}\right) (z(\tau )- z(t-\eta ))\rhoinf(t-\tau) d\tau \leq 0.
\end{aligned}
$$
Factorizing the difference $z(t-\eta)-z(t)$ and dividing by $\eta$ leads to write :
$$
\begin{aligned}
(\vinf -\dz(t) -&\mu_\infty(t)+ 2 \mu_\infty(t-\tau_1)) \frac{(z(t-\eta)-z(t))}{\eta} \\
&+ \underbrace{\frac{2}{\eta}\left( \int_{\tau_1}^{\tau_2} + \int_{t-\eta}^{t}\right) (z(\tau )- z(t-\eta ))\rhoinf(t-\tau) d\tau }_{I_3} \leq 0
\end{aligned}
$$
As $z(\tau)$ is monotone either on $(\tau_1,\tau_2)$ or on $(t-\eta,t)$, the latter term can be estimated as 
$$
\left|I_3\right| \leq \left| \frac{z(t-\eta)-z(t)}{\eta}\right| \left(\int_{ \tau_1}^{\tau_2} + \int_{t-\eta}^t \rhoinf(t-\tau) d\tau \right) \leq C \nrm{\dz}{L^\infty(\rr)} o_\eta(1) 
$$
since $\tau_2$ tends to $\tau_1$ as $\eta$ tends to zero.
One concludes making $\eta$ tend to zero that
$$
(\vinf -\dz(t) -\mu_\infty(t)+ 2 \mu_\infty(t-\tau_1)) ( - \dz(t)) \leq 0
$$
which we divide by $-\dz(t)$, since it is a positive definite quantity by hypothesis.
This leads to
$$
\underbrace{\vinf - \int_{t-\tau_1}^t \rhoinf(a) da }_{I_4(t)} + \mu_\infty(t-\tau_1) \leq \dz.
$$
Then, assuming that $\rhoinf$ is a monotone non-increasing function, shows that $I_5(t) := \int_{t-\tau_1}^t \rhoinf(a) da$ 
is decreasing as well, thus 
$$
I_4(t)=\vinf - I_5(t) \geq \vinf - I_5(\tau_1) = \vinf - \mu_\infty(\tau_1) \geq 0
$$
the latter estimate being true since $\tau_1 < t_1$, which finally
gives that
$$
\mu_\infty(t-\tau_1) \leq \dz.
$$
The latter quantity is strictly positive since $t>t_1>\tau_1$, this leads to a contradiction.  
Indeed, because $\mu_\infty(t_1)=\vinf$ and $\lim_{t \to \infty} \mu_\infty(t) = \mu_\infty > \vinf$,
 there exists an open  set $M \subset (t_1,\infty)$ of positive measure on which $\rhoinf(a)>0$ for a.e. $a\in M$. Since $\rhoinf(a)$ is decreasing there  exist $a_0 \in  M$ such that $\sup_M \rhoinf \geq \rhoinf(a_0)  >0$.
Take $\delta < a_0-t_1$ which implies that $t \in (t_1,a_0)$ then 
$$
\mu_\infty(t-\tau_1 ) := \int_0^{t-\tau_1} \rhoinf(a) da \geq \rhoinf(a_0)\int_0^{t-\tau_1} da = (t-\tau_1) \rhoinf(a_0) > 0.
$$
%means that $\rhoinf(a)\geq\rhoinf(t_1)>0$ for $a \in (0,t_1)$ as
%$\rhoinf$ is a decreasing function. This shows that
%$$
%\mu_\infty(t-\tau_1) \geq \rhoinf(t_1) (t-\tau_1) > 0
%$$
%as soon as $t>\tau_1$ which is true if $t>t_1$.
Thus $\dz$ cannot be negative definite.

We assume now that for $t\in(t_1,t_1+\delta)$, $\dz(t)>0$. We fix $t$ as above.
Again using \eqref{eq.sec.333}, one obtains~:
$$
\left( \vinf - \dz(t) \right)(z(t_1)-z(t)) + \int_{0}^{t} (z(t)-z(t-a)) \rhoinf(a ) da \leq 
\int_{0}^{t} (z(t_1)-z(t-a)) \rhoinf(a ) da
$$
which transforms into :
$$
(\vinf - \dz(t) -\mu_\infty(t) ) (z(t_1) - z(t )) \leq 0
$$
which leads to 
$$
\dz \leq \vinf - \mu_\infty(t) < 0
$$
which again is a contradiction.
Thus $\dz$ must be zero on a positive neighborhood of $t_1$.

Since both arguments extend to any interval $I \in (t_1,\infty)$ the claim is proved when $\vinf \in (-\mu_\infty,\mu_\infty)$.  
For the particular case when $\vinf = \pm \mu_\infty$, the time $t_1$ such that $\mu_\infty(t)=\pm \vinf$ is infinite. 
Thus \eqref{eq.dz} remains true on $\rr$ if $\vinf=\mu_\infty$ and $\dz(t)=\vinf+\mu_\infty(t)$ if $\vinf=-\mu_\infty$.
This can be rewritten as
$$
\begin{aligned}
z(t) & =  z^0 + \sgn(\vinf) \int_0^t \int_{\tau}^{\infty} \rhoinf(a)da d\tau \\
& = z^0 +
\sgn(\vinf) \left\{ \int_0^t a \rhoinf(a)da + t \int_t^\infty \rhoinf(a)da\right\} .
\end{aligned}
$$
\end{proof}

\begin{corom}
	Under the same hypotheses as above, but if $\vinf \notin [-\mu_\infty,\mu_\infty]$, then 
	$$
		z(t)  = z^0 +  
		\int_{0}^{t} \left( \vinf - \sgn(\vinf ) \mu_\infty(\tau) \right)d\tau 
		= z^0 + \gamma t + \sgn(\vinf ) \int_0^t  \int_{\tau}^{\infty} \rhoinf (a) da d \tau 
		%z^0  - \int_0^t a \rhoinf(a) da + t \left\{ \vinf + \sgn(\vinf )  \mu_\infty(t)  \right\}
	$$
\end{corom}

%One can compute $z$ more explicitly :
%\begin{equation}\label{eq.zinf}
%	z(t) = \begin{cases}
%	z^0  + \vinf t - \mu_\infty(t) + \int_0^{t} a \rhoinf(a) da, & \text{if} \; t<t_1 \\
%	z^0 + \vinf t_1 - \mu_\infty(t_1) + \int_0^{t_1} a \rhoinf(a) da, & \text{otherwise} 
%	\end{cases}
%\end{equation}

\subsection{A numerical illustration}
We discretize the previous problem using minimizing movements scheme \cite{AmGiSa}.
We denote 
$R_j := \exp(- j \Delta a)
%\int_{j \Delta a}^{(j+1)\Delta a} \rhoinf(a) da, \quad \forall j \in \NN
$, for $j\in \NN$,
and we approximate the functional $I[w,t] := \int_0^t | w - z(t-a) | \rhoinf(a) da$ by setting
$$
I_n[w] := \Delta a \sum_{j=0}^{n-1} | w - Z^{n-1-j} |R_j,
$$
and the total energy minimized for each time step $n$ reads :
\begin{equation}\label{eq.minimiz.mvts}
	\cE_n(w):= \frac{(w-Z^{n-1})^2}{2 \Delta t} + I_n[w] - \vinf w
\end{equation}
it is a convex functional with respect to $w$ and there exists a unique minimum for each step $n$.
So at each time step $t^{n}=n \Delta t$, we define  $Z^n$ as
$$
Z^n = \argmin\limits_{w \in \RR} \cE_n(w),
$$
One can compare $z$ computed  by this minimization scheme with the 
theoretical formula \eqref{eq.zinf} above. We plot  in fig. \ref{fig.minimis} the result of this computation, where  $\vinf$ is set to $\vinf=0.1$ in the plastic regime cf fig \ref{fig.plastic}, and $\vinf=1.5$ in the kinematic regime (cf fig. \ref{fig.elastic}) with $\mu_\infty=\int_\rr \exp(-a)da=1$.
   
\begin{figure}[ht!]
	\begin{subfigure}{0.4\textwidth}
		\raggedleft
		\begin{tikzpicture}[
		declare function={
			func(\x)= (\x < 0.1053605156578262) * (-0.001+0.1*\x-\x+1-exp(-\x))   +
			(\x >= 0.1053605156578262) * (-0.001-0.9*0.1053605156578262+1-exp(-0.1053605156578262) );	
			%		func(\x)= (\x < 0.1) * (-0.001+0.1*\x-\x*(1-exp(-\x))+1-(1+\x)*exp(-\x))   +
			%				and(\x >= 0.1) * (0.004175535907956274) ;
		}]
		\begin{axis}[small,%
		xlabel={$t$},
		ylabel={$z$},legend style={at={(axis cs:0.2,0.001)},nodes={right}}]
		\addplot[color=blue] 
		table [x expr=\thisrowno{0}, y expr=\thisrowno{1}, col sep=space] {minimisation.gnp};
		\addplot[color=green,domain=0:0.2, mark=*]{func(x)};
		\legend{simulation,prediction}
		\end{axis}
		\end{tikzpicture}
		\caption{Displacement $z(t)$ in the plastic case ($\vinf=0.1<1=\mu_\infty$)}\label{fig.plastic}
	\end{subfigure}
\hspace{1cm}
	\begin{subfigure}{0.4\textwidth}
	\begin{tikzpicture}
		\begin{axis}[small,%
			xlabel={$t$},
	ylabel={$\dz$},legend style={at={(axis cs:10,1)},nodes={right}}]
	\addplot[color=blue] 
	table [x index=0, y index=1, col sep=space] {elastic.gnp};
	\addplot[color=green,mark=*] 
	table [x index=0, y index=2, col sep=space] {elastic.gnp};
		\legend{simulation,prediction}
	\end{axis}
	\end{tikzpicture}
	\caption{Velocity $\dot{z}(t)$ in the kinematic case ($\vinf=1.5>1=\mu_\infty$)}\label{fig.elastic}
\end{subfigure}
	\caption{Numerical simulation using a gradient flow scheme \eqref{eq.minimiz.mvts} associated to \eqref{eq.diff.inclusion.z.intro}}
		\label{fig.minimis}
\end{figure}
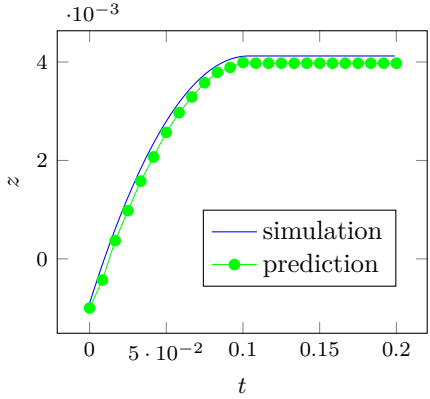
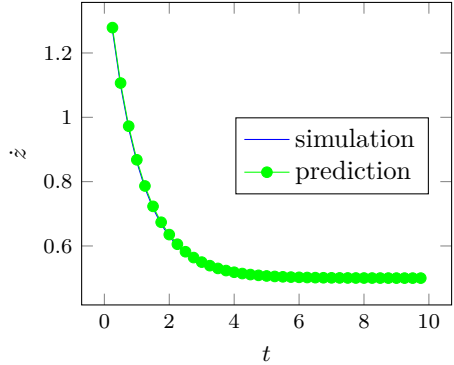

\def\cprime{$'$} \def\cprime{$'$}

\end{document}